\documentclass[a4paper,10pt,twoside]{amsart}
\usepackage{amsfonts, amsmath, amssymb, appendix, mathrsfs}
\usepackage{enumerate}
\usepackage{enumitem}
\usepackage{graphicx}
\numberwithin{equation}{section}
\usepackage{hyperref}

\makeatletter
\DeclareRobustCommand*{\bfseries}{%
  \not@math@alphabet\bfseries\mathbf
  \fontseries\bfdefault\selectfont
  \boldmath
}
\makeatother

\def\wto{\rightharpoonup}
\def\swto{\stackrel{*}{\wto}}
\def\RR{\mathbb{R}}

\def\MM{\mathbb{M}}
\def\NN{\mathbb{N}}
\def\Z{{\mathcal Z}}
\def\H{{\mathcal H}}
\def\Q{\mathcal Q}
\def\h{{W}}

\def\diam{{\rm diam}}
\def\dom{{\rm dom}}
\def\dist{{\rm dist}}

\def\C{{\mathnormal C}}
\def\inte{{\rm int}}
\def\lsp{{\mathrm{l.s.p.}}}
\def\Aff{{\rm Aff}}
\def\eps{\varepsilon}

\newtheorem{theorem}{Theorem}[section]
\newtheorem{lemma}{Lemma}[section]
\newtheorem{proposition}{Proposition}[section]
\newtheorem{corollary}{Corollary}[section]
\newtheorem{definition}{Definition}[section]

\theoremstyle{remark}
\newtheorem{remark}{Remark}[section]

 \numberwithin{equation}{section}

\title[Homogenization of singular integrals in $W^{1,\infty}$]{Homogenization of singular integrals  in ${W}^{1,\infty}$}

\author[Omar Anza Hafsa \& Jean-Philippe Mandallena]{Omar Anza Hafsa \& Jean-Philippe Mandallena}
\address{UNIVERSITE MONTPELLIER II, UMR-CNRS 5508, Place Eug\`ene Bataillon, 34095 Montpellier, France.}

\email{Omar.Anza-Hafsa@univ-montp2.fr}

\address{UNIVERSITE DE NIMES, Site des Carmes, Place Gabriel P\'eri, 30021 N\^\i mes, France.}

\email{jean-philippe.mandallena@unimes.fr}
\begin{document}
\begin{abstract} A periodic homogenization result of nonconvex integral functionals in the vectorial case with convex bounded constraints on gradients is proved. The class of integrands considered have singular behavior near the boundary of the convex set of the constraints. We apply the result to the case of periodic homogenization in hyperelasticity for bounded gradients of deformations.
\end{abstract}
\maketitle
\section{Introduction and main results} Let $m,d\ge 1$ be two integers. Let $\Omega\subset\RR^d$ be a nonempty open bounded domain with Lipschitz boundary. We consider the homogenization problem of integral functionals $I_\eps: W^{1,\infty}(\Omega;\RR^m)\to [0,+\infty]$ given by
\begin{equation}\label{I_eps}
I_\eps(u)=\int_\Omega \h\left(\frac{x}{\eps},\nabla u(x)\right)dx 
\end{equation}
where $\eps>0$ is a (small) parameter. The integrand $\h:\RR^d\times\MM^{m\times d}\to [0,+\infty]$ is Borel measurable, $Y$-periodic with respect to its first variable with $Y=]0,1[^d$, with $\MM^{m\times d}$ denotes the set of $m\times d$ matrices. For any $x\in\RR^d$, we denote by $\dom\h(x,\cdot)$ the effective domain of $\h(x,\cdot)$, i.e., $\dom\h(x,\cdot)=\{\xi\in\MM^{m\times d}:\h(x,\xi)<+\infty\}$. We are interested in integrands satisfying $\dom\h(x,\cdot)\subset\overline{\C}$ for all $x\in\RR^d$, where $\overline{\C}$ is the closure of a convex bounded set with nonempty interior. Since the boundedness of $\C$, the sequential weak$^\ast$ convergence in $W^{1,\infty}(\Omega;\RR^m)$ is the natural convergence for the homogenization of integrals \eqref{I_eps}.

For scalar problems (when $\min\{d,m\}=1$), periodic homogenization problems with convex constraints were studied by several authors, see for instance \cite{carbone-dearcangelis02}. The classical periodic homogenization results for integrals \eqref{I_eps} in the vectorial  case (when $\min\{d,m\}\ge 2$) require polynomial growth conditions on the integrands, which do not allow us to deal with constraints on gradients. However in a recent paper \cite{oah-jpm-Leghmizi09}, motivated by taking gradients constraints arising in hyperelasticity into account, we proved a homogenization result with constraints of type $\det\nabla u\not=0$. With the same motivations and drawing conditions on the determinant of gradients in hyperelasticity, we consider assumptions (see Subsection \ref{mresults}) which allow singular behavior of $\h$ near the boundary $\partial\C$ of $\C$, of type
\[
\lim_{\xi\to \partial\C}\h(\cdot,\xi)=+\infty.
\]
In this perspective we give in Section \ref{sec5} an application to homogenization problems satisfying some natural conditions on determinant of gradients.
\subsection{Main results}\label{mresults} Let $\C\subset\MM^{m\times d}$ be a bounded convex set with nonempty interior. To simplify our statements we assume throughout the paper that $0\in\inte\C$, where $\inte\C$ is the interior of $\C$. Let $\h:\RR^d\times\MM^{m\times d}\to [0,+\infty]$ be a Borel measurable function $Y$-periodic with respect to its first variable, such that $\dom\h(x,\cdot)\subset\overline{\C}$ for all $x\in\RR^d$. We consider the following assertions:
\begin{enumerate}
\item[(H$_1$)]\label{H1} $\h$ is radially uniformly upper semicontinuous on its domain, i.e., for every $\eps>0$ there exists $\eta>0$ such that for every $x\in \RR^d$, every $\xi\in\dom\h(x,\cdot)$ and every $t\in [0,1[$
\[
1-t\le\eta\implies\h(x,t\xi)\le\h(x,\xi)+\eps;
\]
\item[(H$_2$)]\label{H2} $\h$ is locally bounded in $\inte\C$, i.e., $\sup\{\h(\cdot,\xi):\xi\in K\}\in L^\infty_{\rm loc}(\RR^d)$ for all compact sets $K\subset \inte\C$;
\item[(H$_3$)]\label{H3} $\h$ is singular at $\partial\C$, i.e., for every $s>0$ there exists a compact set $K_s\subset\inte\C$ such that for every $x\in\RR^d$
\[
\inf\left\{\h(x,\xi):\xi\in \overline{\C}\setminus K_s\right\}\ge s.
\]
\end{enumerate}
\begin{remark}
If (H$_2$) holds then $\inte\C\subset\dom\h(x,\cdot)\subset\overline{\C}$ a.e. in $\RR^d$. If (H$_2$) and H$_3$) hold then $\dom\h(x,\cdot)=\inte\C$ a.e. in $\RR^d$.
\end{remark}

We state some properties of convex sets.\\
\noindent {\em Line segment principle {\rm (}l.s.p.{\rm )}} \cite[Theorem 2.33]{rockafellar-wets98}:  Let $\C\subset\MM^{m\times d}$ be a bounded convex set with $0\in\inte\C$. Then
\[
 \inte{\overline{\C}}=\inte{\C}, \quad\overline{\inte\C}=\overline{\C},\; \hbox{ and } \;t \overline{\C}\subset \inte\C\hbox{ for all }t\in [0,1[.
\]

We say that {\em $\{I_\eps\}_{\eps>0}$ $\Gamma$-converges to $I_{\rm hom}:W^{1,\infty}(\Omega;\RR^m)\to [0,+\infty]$ with respect to $L^1(\Omega;\RR^m)$-convergence as $\eps\to 0$}, if for every $u\in W^{1,\infty}(\Omega;\RR^m)$
\[
\Gamma\mbox{-}\liminf_{\eps\to 0} I_{\eps}(u)\ge I_{\rm hom}(u)\quad\mbox{ and }\quad\Gamma\mbox{-}\limsup_{\eps\to 0} I_{\eps}(u)\le I_{\rm hom}(u),
\]
where
\begin{align*}
\Gamma\mbox{-}\liminf_{\eps\to 0} I_{\eps}(u)&=\inf\left\{\liminf_{\eps\to 0}I_{\eps}(u_\eps):u_\eps\to u \mbox{ in }L^1\right\},\\
\Gamma\mbox{-}\limsup_{\eps\to 0} I_{\eps}(u)&=\inf\left\{\limsup_{\eps\to 0}I_{\eps}(u_\eps):u_\eps\to u \mbox{ in }L^1\right\}.
\end{align*}

Define $\H\h:\MM^{m\times d}\to [0,+\infty]$ by
\[
\H\h(\xi)=\inf_{n\ge 1}\frac{1}{n^d}\inf\left\{\int_{nY} \h(x,\xi+\nabla\phi(x))dx:\phi\in W^{1,\infty}_0(nY;\RR^m)\right\}.
\]
Here is the main result of the paper. 
\begin{theorem}\label{main result 1} Assume that {\rm (${\rm H}_1$)}, {\rm (${\rm H}_2$)} and {\rm (${\rm H}_3$)} hold. Then $\{I_\eps\}_{\eps>0}$ $\Gamma$-converges to $I_{\rm hom}$ with respect to $L^1(\Omega;\RR^m)$-convergence as $\eps\to 0$, where for every $u\in W^{1,\infty}(\Omega;\RR^m)$
\begin{equation*}
I_{\rm hom}(u)=\int_\Omega\h_{\rm hom}(\nabla u(x))dx 
\end{equation*}
and
\[
\h_{\rm hom}=\overline{\H\h},
\]
where $\overline{\H\h}$ is the lower semicontinuous envelope of $\H\h$.
\end{theorem}

We have also a version of Theorem \ref{main result 1} with Dirichlet boundary conditions. Given any bounded open set $D\subset\RR^d$ with $|\partial D|=0$, we set $W^{1,\infty}_0(D;\RR^m)=\{\varphi\in W^{1,\infty}(D;\RR^m):\varphi=0\hbox{ on }\partial D\}$. Define $J_\eps: W^{1,\infty}(\Omega;\RR^m)\to [0,+\infty]$ by
\[
J_\eps(u)=\left\{\begin{array}{cl}
\displaystyle\int_\Omega \h\left(\frac{x}{\eps},\nabla u(x)\right)dx &\hbox{ if }u\in W^{1,\infty}_0(\Omega;\RR^m)\\
+\infty & \hbox{ if }u\in W^{1,\infty}(\Omega;\RR^m)\setminus W^{1,\infty}_0(\Omega;\RR^m).
\end{array}
\right.
\]

\begin{theorem}[Dirichlet boundary conditions]\label{main result 1bis} Assume that {\rm (${\rm H}_1$)}, {\rm (${\rm H}_2$)} and {\rm (${\rm H}_3$)} hold. Then $\{J_\eps\}_{\eps>0}$ $\Gamma$-converges to $J_{\rm hom}$ with respect to $L^1(\Omega;\RR^m)$-convergence as $\eps\to 0$, where
\[
J_{\rm hom}(u)=\left\{
\begin{array}{cl}
  \displaystyle\int_\Omega\h_{\rm hom}(\nabla u(x))dx & \hbox{ if }\;u\in W^{1,\infty}_0(\Omega;\RR^m)\\
  +\infty & \hbox{ if }\; u\in W^{1,\infty}(\Omega;\RR^m)\setminus W^{1,\infty}_0(\Omega;\RR^m)
\end{array}
\right.
\]
and
$
\h_{\rm hom}=\overline{\H\h}.
$
\end{theorem}

We say that a function $g:\MM^{m\times d}\to [0,+\infty[$ is {\em quasiconvex} at $\xi\in\MM^{m\times d}$ if it is Borel measurable and 
\[
g(\xi)=\inf\left\{\int_Y g(\xi+\nabla \phi(x))dx: \phi\in W^{1,\infty}_0(Y;\RR^m)\right\},
\]
where $Y=]0,1[^d$ is the unit cube in $\RR^d$. If $g$ is quasiconvex at every $\xi\in\MM^{m\times d}$ then $g$ is said quasiconvex. If $g$ is quasiconvex then it is continuous (see for instance Dacorogna \cite{dacorogna08}).

Let us define by $\Q f:\MM^{m\times d}\to [0,+\infty]$ the {\em quasiconvex envelope} of $f:\MM^{m\times d}\to [0,+\infty]$ given by
\[
\Q f(\xi)=\sup\left\{g(\xi):g:\MM^{m\times d}\to [0,+\infty[ \hbox{ is quasiconvex and }g\le f\right\}.
\]
Note that $\Q f$ is lower semicontinuous as pointwise supremum of continuous functions and satisfies for all $\xi\in\MM^{m\times d}$
\begin{equation*}\label{morrey inequality for Qf}
\Q f(\xi)=\inf\left\{\int_Y \Q f(\xi+\nabla \phi(x))dx:\phi\in W^{1,\infty}_0(Y;\RR^m)\right\}.
\end{equation*}

We denote  the space of continuous piecewise affine functions from $D$ to $\RR^m$ by $\Aff(D;\RR^m)$, i.e., $\varphi\in\Aff(D;\RR^m)$ if and only if $\varphi$ is continuous and there exists a finite family $\{D_i\}_{i\in I}$ of open disjoint subsets of $D$ such that $|\partial D_i|=0$ for all $i\in I$, $|D\setminus \cup_{i\in I} D_i|=0$ and for every $i\in I$, $\nabla \varphi\equiv \xi_i$ in $D_i$ with $\xi_i\in\MM^{m\times d}$, and we set $\Aff_0(D;\RR^m)=\{\varphi\in\Aff(D;\RR^m):\varphi=0\hbox{ on }\partial D\}$.

Let $f:\MM^{m\times d}\to [0,+\infty]$ be a Borel measurable. Define $\Z f:\MM^{m\times d}\to [0,+\infty]$ by
\begin{equation}\label{defzf}
\Z f(\xi)=\inf\left\{\int_Y f(\xi+\nabla \psi(x)):\psi\in \Aff_0(Y;\RR^m)\right\}.
\end{equation}
\begin{lemma}[\cite{oah-jpm08a}]\label{quasifinite} If $\Z f$ is finite then $\Q f=\Z f$.
\end{lemma}
\begin{lemma}[\cite{fonseca88}]\label{fonseca}
The function $\Z f$ is continuous in $\inte(\dom\Z f)$.
\end{lemma}
The following representation result is the main ingredient for the proof of Theorem~\ref{main result 1}. More precisely, items \ref{mainresult21}, \ref{mainresult22} and \ref{mainresult23} are used in the proof of the $\Gamma$-$\liminf_{\eps\to 0}I_\eps$ (see Subsection \ref{gliminf}), and \ref{mainresult24} is used in $\Gamma$-$\limsup_{\eps\to 0}I_\eps$ (see Subsection \ref{glimsup}).
\begin{theorem}\label{main result 2} Assume that {\rm (${\rm H}_1$)}, {\rm (${\rm H}_2$)} and {\rm (${\rm H}_3$)} hold. There exists a nondecreasing sequence of functions $\{\h_n\}_{n\in\NN^\ast}$ such that for any $n\in\NN^\ast$ the function $\h_n:\RR^d\times \MM^{m\times d}\to [0,+\infty[$ is Borel measurable, $Y$-periodic with respect to its first variable and satisfies
\begin{enumerate}[label={\rm (\roman{*})}, ref={\rm (\roman{*})}]
\item there exists $\alpha_n>0$ such that $\h_n(x,\xi)\le \alpha_n\left(1+\vert\xi\vert\right)$ a.e. in $\RR^d$ and $\xi\in\MM^{m\times d}$\label{mainresult21};
\item $\h_n(x,\xi)\le \h(x,\xi)$ for all $x\in\RR^d$ and all $\xi\in\MM^{m\times d}$\label{mainresult22};
\item $\sup_{n\in\NN^\ast}\H\h_n(\xi)=\overline{\H\h}(\xi)$ for all $\xi\in\MM^{m\times d}$, where $\overline{\H\h}$ is the lower semicontinuous envelope of $\H\h$\label{mainresult23};
\item \label{mainresult24} it holds $\overline{\H\h}=\Q\H\h=\widehat{\Z\H\h}$ where
\[
\widehat{\Z\H\h}(\xi)=\left\{\begin{array}{ll}
\displaystyle \Z\H\h(\xi)& \hbox{ if }\xi\in\inte\C\\ \\
\displaystyle\lim_{[0,1[\ni t\to 1}\Z\H\h(t\xi)&\hbox{ if }\xi\in\partial\C\\\\
+\infty &\hbox{ otherwise.}
\end{array}
\right.
\]
\end{enumerate}
\end{theorem}
Let $h:\MM^{m\times d}\to [0,+\infty]$ be a Borel measurable function. We say that $h$ is~{\em $p$-sup-quasiconvex} if there exist $p\in [1,+\infty[$ and a nondecreasing sequence $\left\{h_n\right\}_{n\in\NN}$, $h_n:\MM^{m\times d}\to [0,+\infty[$ such that
\begin{enumerate}[label={\rm (\roman{*})}, ref={\rm (\roman{*})}]
\item $h_n$ is quasiconvex for all $n\in\NN$;
\item for every $n\in\NN$ there exists $\alpha_n>0$ such that $h_n(\xi)\le \alpha_n(1+\vert\xi\vert^p)$ for all $\xi\in\MM^{m\times d}$;
\item for every $\xi\in\MM^{m\times d} $ 
\[
\sup\{h_n(\xi):n\in\NN\}=h(\xi).
\]
\end{enumerate}
It is easy to see that if $h$ is $p$-sup-quasiconvex then it is lower semicontinuous as pointwise supremum of continuous functions and satisfies
\begin{equation*}\label{morrey inequality for g}
h(\xi)=\inf\left\{\int_Yh(\xi+\nabla \phi(x))dx:\phi\in W^{1,\infty}_0(Y;\RR^m)\right\}
\end{equation*}
for all $\xi\in\MM^{m\times d}$. Tartar \cite{tartar93} has shown that there exist quasiconvex functions which are not {$p$-sup-quasiconvex} for any $p$. From Theorem \ref{main result 2}, we deduce
\begin{corollary}\label{co}Under the same assumptions of Theorem \ref{main result 2}, $\overline{\H\h}$ is $1$-sup-quasi\-convex.
\end{corollary}
\begin{proof} Since Theorem \ref{main result 2} \ref{mainresult21}, for every $\xi\in\MM^{m\times d}$ and every $n\in\NN^*$
\[
\H\h_n(\xi)\le \int_Y\h_n(x,\xi)dx\le\alpha_n(1+\vert\xi\vert).
\]
From Theorem \ref{BraidesTheorem}, each $\H\h_n$ is quasiconvex and by Theorem \ref{main result 2} \ref{mainresult23}, it follows that $\overline{\H\h}$ is $1$-sup-quasiconvex.
\end{proof}
\subsection{Organization of the paper} In Section \ref{sec5}, we show an application of our results for homogenization problems in hyperelasticity with bounded gradients of deformations. Section \ref{semihom} is concerned with ``semi-homogenization" of periodic integrals with linear growth in $W^{1,\infty}(\Omega;\RR^m)$. This provides us, among other things, that $\H f$ is quasiconvex when $f$ is of linear growth. The proof of Theorem \ref{main result 1} is divided into two steps. The proof of the $\Gamma\mbox{-}\liminf_{\eps\to 0}I_\eps$ follows easily by using Theorem \ref{main result 2} and Theorem \ref{BraidesTheorem}. The proof of the $\Gamma\mbox{-}\limsup_{\eps\to 0}I_\eps$ is more classical, we use approximation result for Lipschitz functions by continuous affine piecewise, and one lemma on $\Z f$ which is developed in our previous papers (see for instance \cite{oah-jpm08a}). In Section \ref{lineseg} we are mainly concerned with establishing some properties of $\H\h$ and $\Z\H\h$. The proof of Theorem \ref{main result 2} is achieved by adapting some arguments of M\"uller \cite{muller99} to our case.
\section{Periodic homogenization in hyperelasticity for bounded gradients of deformations}\label{sec5}
The goal of this section is to show that, by using our results, we can take account of the natural determinant conditions of hyperelasticity in homogenization problems with bounded gradients of deformations. 

We assume that $m=d$. Let $A\in L^\infty_{\rm loc}(\RR^d)$ be an $Y$-periodic function satisfying for some $c>0$
\[
\mathop{{\rm essinf}}_{x\in\RR^d}A(x)\ge c.
\]
Define $f:\MM^{d\times d}\to [0,+\infty]$ by 
\[
f(\xi)=\left\{\begin{array}{cl}
g(\xi)+h(\vert\xi-I\vert)& \hbox{ if }\xi\in B(I)\\
+\infty&\hbox{ otherwise,}
\end{array}\right.
\]
where $B(I)=\{\xi\in\MM^{d\times d}:\vert\xi-I\vert<1\}$, and 
\begin{enumerate}[label=$\diamond$, ref=$\diamond$]
\item $g\in C(\overline{B}(I))$ is continuous and $g\ge 0$;
\item $h: [0,1[\to {\RR}_+$ is convex, and $h(t)\ge C\left(\frac{1}{1-t^\alpha}-1\right)$ for all $t\in\RR_+$, for some $\alpha>0,C>0$.
\end{enumerate}
The open ball $B(I)$ is contained in $\MM^{d\times d}_{+}=\{\xi\in\MM^{d\times d}:\det\xi>0\}$. Indeed, by using the infinite sum $\zeta=\sum_{i=0}^{+\infty}(I-\xi)^i$, it is easy to show that if $\vert\xi-I\vert<1$ then $I-(I-\xi)=\xi$ is invertible with inverse $\zeta$. Then with the same arguments, $I-t(I-\xi)$ is invertible for all $t\in [0,1[$. The continuous function $[0,1]\ni t\mapsto \det(I-t(I-\xi))=\alpha(t)$ is such that $\alpha(0)=1$ and $\alpha(t)\not=0$ for all $t\in [0,1]$, it follows that $\alpha(t)>0$ for all $t\in [0,1]$, and in particular $\det\xi>0$.

We consider the integrand $\RR^d\times \MM^{d\times d}\ni(x,\xi)\mapsto\h(x,\xi)=A(x)f(\xi)$ with $A$ and $f$ as above. Proposition \ref{nim} below shows that if $\h$ is the stored energy density of a periodically heterogeneous material, it satisfies the natural requirements of hyperelasticity, i.e., non-interpenetration of the matter and the requirement of infinite amount of energy to compress finite volume into zero volume. However, only ``small" gradients of deformations around equilibrium configuration are allowed. By choosing $\alpha$ very large, any stored energy density $g$ is almost not modified around the equilibrium configuration. Thus it is possible to consider a large range of nonlinear models by specifying $g$ and by adding the ``singular" perturbation $h(\vert\cdot-I\vert)$.
\begin{proposition}\label{nim} The function $\h$ is a Carath\'eodory integrand and satisfies for every $x\in\RR^d$
\[
\dom \h(x,\cdot)\subset\MM^{d\times d}_+\quad\mbox{ and }\quad\lim_{\begin{subarray}{c} 
\det\xi\to 0\\ 
\xi\in\dom\h(x,\cdot) 
\end{subarray}}\h(x,\xi)=+\infty.
\]
\end{proposition}
\begin{proof} Since $h$ is convex finite in $[0,1[$ and
\[
\lim_{\vert\xi-I\vert\to 1}h(\vert \xi-I\vert)\ge C\left(\lim_{\vert\xi-I\vert\to 1}\frac{1}{1-\vert\xi-I\vert^\alpha}-1\right)=+\infty,
\]
it follows that the function 
\[
\MM^{d\times d}\ni\xi\mapsto\left\{
\begin{array}{cl}
h(\vert\xi-I\vert)&\mbox{ if }\xi\in B(I)\\
+\infty&\mbox{ otherwise, }
\end{array}
\right.
\]
is continuous. Thus $\h$ is a Carath\'eodory integrand.

We have $\dom \h(x,\cdot)=B(I)\subset \MM^{d\times d}_+$ for all $x\in\RR^d$. Let $\{\xi_n\}_{n\in\NN}\subset B(I)$ such that $\det\xi_n\to 0$ as $n\to +\infty$. There exists a subsequence $\{\xi_{\sigma(n)}\}_{n\in\NN}$ and $\xi\in \overline{B}(I)$ such that $\xi_{\sigma(n)}\to \xi$ as $n\to +\infty$. By continuity, $\det\xi=0$ which implies $\vert\xi-I\vert=1$. Thus for every $x\in\RR^d$
\[
\liminf_{n\to +\infty}\h(x,\xi_n)\ge cC\lim_{n\to +\infty}\frac{1}{1-\vert \xi_n-I\vert^\alpha}-cC=+\infty,
\]
hence
$\displaystyle
\lim_{\begin{subarray}{c} 
\det\xi\to 0\\ 
\xi\in\dom\h(x,\cdot) 
\end{subarray}}\h(x,\xi)=+\infty.
$
\end{proof}
\begin{theorem}\label{hyperhomog} Define $E_\eps: W^{1,\infty}(\Omega;\RR^d)\to [0,+\infty]$ by
\[
E_\eps(\varphi)=\left\{\begin{array}{cl}
\displaystyle\int_\Omega \h\left(\frac{x}{\eps},\nabla \varphi(x)\right)dx &\hbox{ if }\varphi\in l_I+W^{1,\infty}_0(\Omega;\RR^d)\\
+\infty & \hbox{ if }\varphi\in W^{1,\infty}(\Omega;\RR^d)\setminus \left(l_I+W^{1,\infty}_0(\Omega;\RR^d)\right).
\end{array}
\right.
\]
The functionals $\{E_\eps\}_{\eps>0}$ $\Gamma$-converges to $E_{\rm hom}$ with respect to $L^1(\Omega;\RR^d)$-convergence as $\eps\to 0$, where
\[
E_{\rm hom}(\varphi)=\left\{
\begin{array}{cl}
  \displaystyle\int_\Omega\h_{\rm hom}(\nabla \varphi(x))dx & \hbox{ if }\;\varphi\in l_I+W_0^{1,\infty}(\Omega;\RR^d)\\
  +\infty & \hbox{ if }\; \varphi\in W^{1,\infty}(\Omega;\RR^d)\setminus \left(l_I+W_0^{1,\infty}(\Omega;\RR^d)\right).
\end{array}
\right.
\]
and
\[
\h_{\rm hom}(\xi)=\left\{\begin{array}{cl}
\Z\H\h(\xi)& \hbox{ if }\xi\in B(I)\\ 
+\infty &\hbox{ otherwise.}
\end{array}
\right.
\]
Moreover $\h_{\rm hom}$ is continuous and satisfies
\[
B(I)=\dom \h_{\rm hom}\subset\MM^{d\times d}_+\quad\mbox{ and }\quad \lim_{\begin{subarray}{c} 
\det\xi\to 0\\ 
\xi\in\dom\h_{\rm hom} 
\end{subarray}}\h_{\rm hom}(\xi)=+\infty.
\]
\end{theorem}
\begin{proof}
We need the following result for the proof of Theorem \ref{hyperhomog}.
\begin{lemma}\label{asswo} Let $\h_0:\Omega\times \MM^{d\times d}\to [0,+\infty]$ be defined by $\h_0(x,\xi)=\h(x,\xi+I)$ for all $(x,\xi)\in \Omega\times \MM^{m\times d}$. Then $\h_0$ satisfies {\rm (H$_1$)}, {\rm (H$_2$)} and {\rm (H$_3$)}.
\end{lemma}
\begin{proof}
Let us prove that $\h_0$ satisfies {\rm (H$_1$)}. We have $\dom\h_0(x,\cdot)=B(0)$ for all $x\in\RR^d$. Let $\eps>0$. Since $g$ is uniformly continuous on $\overline{B}(I)$, there exists $\delta>0$ such that for every $\xi\in B(0)$ and every $t\in [0,1[$
\begin{equation}\label{grusc}
1-t\le \delta \Longrightarrow g(t(\xi+I))\le g(\xi+I)+\eps.
\end{equation}
Set 
\[
\eta=\left\{\begin{array}{ll}
\min\{\delta,\frac{\eps}{h(0)}\}&\mbox{ if }h(0)>0\\
\delta&\mbox{ if }h(0)=0.
\end{array}
\right.
\]

Since $h$ is convex, for every $\xi\in B(0)$ and every $t\in[0,1[$ such that $1-t\le \eta$
\begin{equation}\label{hrusc}
h(\vert t\xi\vert)\le t h(\vert \xi\vert)+(1-t)h(0)\le h(\vert\xi\vert)+\eps.
\end{equation}
Using \eqref{grusc} and \eqref{hrusc}, $\h_0$ satisfies {\rm (H$_1$)}.

Let $K\subset B(0)$ be a compact. The function $\h_0$ satisfies {\rm (H$_2$)}, since $g(\cdot+I)$ is bounded on $K$, $h(\vert\cdot\vert)$ is convex and finite then bounded on $K$.

Let $\{t_n\}_{n\ge 1}\subset [0,1[$ be such that $t_n\to 1$ as $n\to +\infty$. It holds 
\[
\inf_{\xi\in B(0)\setminus t_n \overline{B}(0)}\h_0(x,\xi)\ge cC\left(\inf_{\xi\in B(0)\setminus t_n \overline{B}(0)}\frac{1}{1-\vert\xi\vert^\alpha}-1\right)\ge cC\left(\frac{1}{1-t_n^\alpha}-1\right).
\]
Since $\lim_{n\to+\infty}(1-t_n^\alpha)^{-1}=+\infty$, it is easy to find a subsequence $\{t_{\sigma(n)}\}_{n\ge 1}\subset [0,1[$ such that for every $n\ge 1$
\[
\frac{1}{1-t_{\sigma(n)}^\alpha}\ge \frac{n}{cC}+1,
\]
thus, by taking account of Lemma \ref{decomposition} \ref{decomposition5}, {\rm (H$_3$)} holds.
\end{proof}
By translation, $E_\eps(\cdot)=E_\eps^0(\cdot+l_I)$ for all $\eps>0$, where
\[
E_\eps^0(u)=\left\{
\begin{array}{cl}
  \displaystyle\int_\Omega\h_0\left(\frac{x}{\eps},\nabla u(x)\right)dx & \hbox{ if }\;u\in W_0^{1,\infty}(\Omega;\RR^d)\\
  +\infty & \hbox{ if }\; u\in W^{1,\infty}(\Omega;\RR^d)\setminus W_0^{1,\infty}(\Omega;\RR^d).
\end{array}
\right.
\]
Using Lemma \ref{asswo} together with Theorem \ref{main result 1bis}, we obtain 
\[
\Gamma\mbox{-}\lim_{\eps\to 0}E_\eps(\cdot)=\Gamma\mbox{-}\lim_{\eps\to 0}E_\eps^0(\cdot+l_I)=E^{0}_{\rm hom}(\cdot+l_I).
\]
Set $E_{{\rm hom}}(\cdot)=E^{0}_{\rm hom}(\cdot+l_I)$. By Corollary \ref{cor4}, $\h_{\rm hom}$ is quasiconvex, continuous and for every $\xi\in\MM^{d\times d}$
\[
\h_{\rm hom}(\xi)=\h_{0,{\rm hom}}(\xi-I)=\left\{
\begin{array}{cl}
{\Z\H\h_0}(\xi-I)&\mbox{ if }\xi\in B(I)\\
+\infty&\mbox{ if }\xi\notin B(I)
\end{array}
\right.
\]
Using the same arguments as in the proof of Proposition \ref{nim}, we obtain
\[
\lim_{\begin{subarray}{c} 
\det\xi\to 0\\ 
\xi\in \dom\h_{\rm hom} 
\end{subarray}}\h_{\rm hom}(\xi)=+\infty\quad\mbox{ and }\quad\dom \h_{\rm hom}\subset\MM^{d\times d}_+.
\]
\end{proof}
\begin{remark} Let $\{\varphi_\eps\}_{\eps>0}\subset W^{1,\infty}(\Omega;\RR^m)$ be a sequence satisfying~$\sup_{\eps>0}I_\eps(\varphi_\eps)<+\infty$. Using similar arguments as Subsection \ref{gliminf} (take Remarks \ref{notconnected} into account if $\Omega$ is not connected), it is easy to deduce that there exists a subsequence such that $\varphi_\eps\to \varphi$ in $L^1(\Omega;\RR^m)$. By Theorem  \ref{hyperhomog}, it follows that $\varphi\in l_I+W^{1,\infty}_0(\Omega;\RR^m)$ is a minimizer of $I_{\rm hom}$ and $\det\nabla\varphi>0$ a.e. in $\Omega$. (From \cite[Proposition 3.6]{pourciau83}, we also have that $\varphi:\overline{\Omega}\to \varphi(\overline{\Omega})$ is an homeomorphism.) 
\end{remark}
\section{Semi-homogenization of integrals with linear growth in $W^{1,\infty}(\Omega;\RR^m)$}\label{semihom} Here the goal is to prove $\Gamma$-convergence of periodic integrals with linear growth in $W^{1,\infty}(\Omega;\RR^m)$ with respect to weak$^\ast$ convergence in $W^{1,\infty}(\Omega;\RR^m)$. Since linear growth on the integrands are not compatible with the natural coercivity assumptions in $W^{1,\infty}(\Omega;\RR^m)$ (the integrands should be infinite outside a bounded subset of $\MM^{m\times d}$ to expect to have compactness of $\eps$-minimizing sequences) we have a non complete $\Gamma$-convergence result. 

For any $F\in\MM^{m\times d}$, we denote by $l_F$ the linear function given by $l_F(x)=Fx$ for all $x\in\RR^d$.
Let $\Omega\subset\RR^d$ be a bounded open set. Let $f:\RR^d\times\MM^{m\times d}\to [0,+\infty]$ be a Borel measurable function. Consider the following assertions:
\begin{enumerate}[label={\rm (\roman{*})}, ref={\rm (\roman{*})}]
\item\label{br1} $\RR^{d}\ni x\mapsto f(x,\xi)$ is $Y$-periodic for all $\xi\in\MM^{m\times d}$;
\item\label{br2} there exists $c>0$ such that $f(x,\xi)\le c(1+\vert\xi\vert)$ for all $(x,\xi)\in \RR^d\times\MM^{m\times d}$.
\end{enumerate}
For any $r\in [0,+\infty]$, define $G_r:W^{1,\infty}(\Omega;\RR^m)\to [0,+\infty]$ by
\[
G_r(u)=\inf\left\{\liminf_{\eps\to 0}\int_\Omega f\left(\frac{x}{\eps},\nabla u_\eps\right)dx:u_\eps\stackrel{\ast}{\wto} u\mbox{ in }W^{1,\infty}\mbox{ and }\sup_{\eps>0}\Vert\nabla u_\eps\Vert_\infty\le r\right\}.
\]
For any $r\in [0,+\infty[$, define $H_r:W^{1,\infty}(\Omega;\RR^m)\to [0,+\infty]$ by
\[
H_r(u)=\inf\left\{\limsup_{\eps\to 0}\int_\Omega f\left(\frac{x}{\eps},\nabla u_\eps\right)dx:u_\eps\stackrel{\ast}{\wto} u\mbox{ in }W^{1,\infty}\mbox{ and }\sup_{\eps>0}\Vert\nabla u_\eps\Vert_\infty\le r\right\}.
\]
For any $r\in [0,+\infty]$ and every $\xi\in\MM^{m\times d}$ set
\[
\H_rf(\xi)=\inf_{n\ge 1}\frac{1}{n^d}\inf\left\{\int_{nY} f(x,\xi+\nabla \phi(x))dx:\phi\in W^{1,\infty}_0(nY;\RR^m),\;\Vert\nabla\phi\Vert_\infty\le r\right\}.
\]
Note that $\H_\infty f=\H f$.
\begin{theorem}\label{BraidesTheorem} Assume that \ref{br1} and \ref{br2} hold. For every $u\in W^{1,\infty}(\Omega;\RR^m)$ and every $r\in [0,+\infty]$
\begin{equation}\label{braides1}
G_r(u)\ge \int_\Omega \H f(\nabla u(x))dx;
\end{equation}
For every $u\in \Aff(\Omega;\RR^m)$ and every $r\in [0,+\infty[$
\begin{equation}\label{braides2}
H_{r+\Vert\nabla u\Vert_\infty}(u)\le\int_\Omega  \H_r f(\nabla u(x))dx.
\end{equation}
\end{theorem}
{\noindent{\em Proof of \eqref{braides1}.}} It is enough to prove $G_\infty(u)\ge \int_\Omega \H f(\nabla u(x))dx$ for all $u\in W^{1,\infty}(\Omega;\RR^m)$.

We use the technique of localization and blow-up (see for instance \cite{alvarez-jpm04}). Let $\{\eps\}_{\eps>0}\subset ]0,+\infty[$ be such that $\lim_{\eps\to 0}\eps=0$. Let $\{u_\eps\}_{\eps>0}\subset W^{1,\infty}(\Omega;\RR^m)$ be such that $u_\eps\stackrel{\ast}{\wto} u$ in $W^{1,\infty}(\Omega;\RR^m)$. By compactness imbedding theorem, for a subsequence (not relabeled)
\begin{equation}\label{convunif}
\sup_{\eps>0}\|\nabla u_\eps\|_\infty<+\infty\;\;\mbox{ and }\;\lim_{\eps\to 0}\|u_\eps-u\|_\infty=0.
\end{equation}
Without loss of generality, we can assume \[\liminf_{\eps\to 0}\int_\Omega f\left(\frac{x}{\eps},\nabla u_\eps(x)\right)dx=\lim_{\eps\to 0}\int_\Omega f\left(\frac{x}{\eps},\nabla u_\eps(x)\right)dx<+\infty.\] 
Let $\{\mu_\eps\}_{\eps>0}$ be a sequence of nonnegative Borel measures defined by \[\mu_\eps=f\left(\frac{\cdot}{\eps},\nabla u_\eps(\cdot)\right)dx{\lfloor_{\Omega}}\] where $dx{\lfloor_{\Omega}}$ is the Lebesgue measure on $\Omega$. We have $\sup_{\eps>0}\mu_\eps(\Omega)<+\infty$ and then there exists a subsequence (not relabeled) such that $\mu_\eps\wto \mu$ weak$^\ast$ in the sense of measures. By the Radon-Nikodym decomposition $\mu=gdx{\lfloor_{\Omega}}+\mu_s$ where $g\in L^{1}(\Omega)$ and $\mu_s$ is a nonnegative Borel measure singular with respect $dx{\lfloor_{\Omega}}$. We claim that it is enough to prove $g(\cdot)\ge\H f(\nabla u(\cdot))$ a.e. in $\Omega$, indeed by using Alexandrov theorem
\begin{align*}
\liminf_{\eps\to 0}\int_\Omega f\left(\frac{x}{\eps},\nabla u_\eps(x)\right)dx=\liminf_{\eps\to 0}\mu_\eps(\Omega)&\ge \mu(\Omega)\\ &\ge\int_\Omega g(x)dx\ge \int_\Omega \H f(\nabla u(x))dx.
\end{align*}
Let $B_\rho(x)=x+\rho Y$ be the ball with center $x\in\Omega$ and radius $\rho>0$. By differentiation of measures we have
\begin{equation}\label{diffmeas}
g(x)=\lim_{\rho\to 0}\frac{\mu(B_\rho(x))}{\vert B_\rho(x)\vert}\mbox{ a.e. in }\Omega.
\end{equation}
Fix $x_0\in\Omega$ such that \eqref{diffmeas} is satisfied. By Alexandrov theorem we  have 
\begin{equation}\label{diffmeas1}
g(x_0)=\lim_{\rho\to 0}\lim_{\eps\to 0}\frac{1}{\vert B_\rho(x_0)\vert}\int_{B_\rho(x_0)} f\left(\frac{x}{\eps},\nabla u_\eps(x)\right)dx.
\end{equation}
Let $\rho>0$. Let $u_0\in\Aff(B_\rho(x_0);\RR^m)$ be given by $u_0(x)=u(x_0)+l_{\nabla u(x_0)}(x-x_0)$. Let $\delta\in ]0,1/2[$. Let $u_\eps^\delta=\phi_\delta u_\eps+(1-\phi_\delta)u_0\in l_{\nabla u(x_0)}+W^{1,\infty}_0(B_\rho(x_0);\RR^m)$ where $\phi_\delta\in C_c^\infty(B_\rho(x_0))$ satisfies $\phi_\delta=1$ on $B_{\delta\rho}(x_0)$, $\phi_\delta=0$ on $B_{\rho}(x_0)\setminus B_{2\delta\rho}(x_0)$ and $\|\nabla\phi_\delta\|_\infty\le \frac{c^\prime}{\delta\rho}$ for some $c^\prime>0$ (which does not depend on $\delta$ and $\rho$). By an easy computation we have for every $\eps>0$
\begin{align*}
\int_{B_\rho(x_0)} f\left(\frac{x}{\eps},\nabla u_\eps^\delta(x)\right)dx&\le c\vert B_{2\delta\rho}(x_0)\setminus B_{\delta\rho}(x_0)\vert\left(1+\sup_{\eps>0}\|\nabla u_\eps\|_\infty+\vert\nabla u(x_0)\vert\right)\\&+cc^\prime\vert B_{2\delta\rho}(x_0)\setminus B_{\delta\rho}(x_0)\vert\|u_\eps-u_0\|_\infty\\&+c\vert B_{\rho}(x_0)\setminus B_{2\delta\rho}(x_0)\vert(1+\vert\nabla u(x_0)\vert)\\&+\int_{B_\rho(x_0)} f\left(\frac{x}{\eps},\nabla u_\eps(x)\right)dx.
\end{align*}
Taking into account \eqref{convunif}
\begin{align}
\label{form1}\liminf_{\rho\to 0}\liminf_{\delta\to 1}\liminf_{\eps\to 0}&\frac{1}{\vert B_\rho(x_0)\vert}\int_{B_\rho(x_0)} f\left(\frac{x}{\eps},\nabla u_\eps^\delta(x)\right)dx\\&\le \lim_{\rho\to 0}\lim_{\eps\to 0}\frac{1}{B_\rho(x_0)}\int_{B_\rho(x_0)} f\left(\frac{x}{\eps},\nabla u_\eps(x)\right)dx=g(x_0)\nonumber.
\end{align}
Using subadditive arguments (see for instance \cite[Appendix B]{alvarez-jpm02})
\begin{align*}
\H f(\nabla u(x_0))
&=\lim_{\eps\to 0}\inf_{\phi\in W^{1,\infty}_0\left(\frac{1}{\eps}B_\rho(x_0);\RR^m\right)}\frac{1}{\vert\frac{1}{\eps}B_\rho(x_0)\vert}\int_{\frac{1}{\eps}B_\rho(x_0)} f(x,\nabla u(x_0)+\nabla \phi(x))dx\\
&\le\liminf_{\eps\to 0}\frac{1}{\vert B_\rho(x_0)\vert}\int_{B_\rho(x_0)} f\left(\frac{x}{\eps},\nabla u_\eps^\delta(x)\right)dx
\end{align*}
We deduce by taking account of \eqref{form1} that $g(x_0)\ge \H f(\nabla u(x_0))$.\\

{\noindent{\em Proof of \eqref{braides2}.}} To have \eqref{braides2}, apply Lemma \ref{homog lemma} below with $\mathfrak D=\MM^{m\times d}$, $U=\Omega$ and $j=f$.\\
The following lemma is extracted from \cite[p.195]{muller87}. 
\begin{lemma}\label{homog lemma} Let $\mathfrak{D}\subset \MM^{m\times d}$ be a set. Let $j:\Omega\times\MM^{m\times d}\to [0,+\infty]$ be a Borel measurable function, $Y$-periodic with respect to its first variable and satisfying $j(\cdot,\xi)\in L^\infty_{\rm loc}(\RR^d)$ for all $\xi\in \mathfrak{D}$. Let $U\subset\Omega$ be an open set. Let $w\in\Aff(U;\RR^m)$ be such that $\nabla w(x)\in\mathfrak{D}$ for almost all $x\in U$. Let $r\in [0,+\infty[$. There exists $\{w_\eps\}_{\eps>0}\subset w+W^{1,\infty}_0(U;\RR^m)$ such that $w_\eps\stackrel{\ast}{\wto} w$ in $W^{1,\infty}(\Omega;\RR^m)$, $\sup_{\eps>0}\Vert\nabla w_\eps\Vert_\infty\le \Vert\nabla w\Vert_\infty+r$, and 
\[
\limsup_{\eps\to 0}\int_U j\left(\frac{x}{\eps},\nabla w_\eps(x)\right)dx\le \int_U\H_r j(\nabla w(x))dx.
\]
Moreover, if $\dom j(x,\cdot)\subset B_R(0)$ for all $x\in \RR^d$ and some $R>0$ then there exists a subsequence (not relabeled) $\{w_\eps\}_{\eps>0}\subset w+W^{1,\infty}_0(U;\RR^m)$ such that $w_\eps\to w$ in $L^{1}(\Omega;\RR^m)$, and 
\[
\limsup_{\eps\to 0}\int_U j\left(\frac{x}{\eps},\nabla w_\eps(x)\right)dx\le \int_U\H j(\nabla w(x))dx.
\]
\end{lemma}
\begin{proof} Let $r\in [0,+\infty[$. Assume that $w$ is such that $\nabla w=\xi$ a.e. in $V$, where $\xi\in{\mathfrak D}$ and $V\subset\Omega$ is an open set. Let $s\in\NN^*$. There exists $n_s\in\NN^*$ and $\phi_s\in W^{1,\infty}_0(nY;\RR^m)$ with $\Vert\nabla \phi_s\Vert_\infty\le r$, be such that
\[
\frac{1}{n_s^d}\int_{nY} j(x,\xi+\nabla\phi_s(x))dx\le\H_r j(\xi)+\frac{1}{s}.
\]
Let $]0,+\infty[\ni\eps_l\to 0$ as $l\to +\infty$. Consider $V_{l,s}=\cup\{\eps_l(z+ n_sY):z\in\mathbb{Z}^d,\;\eps_l(z+ n_sY)\subset V\}$ and define
\[
w_{l,s}(x)=\left\{
\begin{array}{cl}
\displaystyle\xi x+\eps_l\phi_s\left(\frac{x}{\eps_l}\right)  &  \mbox{ if }x\in V_{l,s}    \\
 \xi x &     \mbox{ if }x\in V\setminus V_{l,s}
\end{array}
\right.
\]
where we have denoted again by $\phi_s$ the $n_s Y$-periodic extension in $\RR^d$ of $\phi_s$. We have $w_{l,s}\in l_\xi+W_0^{1,\infty}(V;\RR^m)$, $w_{l,s}\stackrel{\ast}{\wto} \xi$ in $W^{1,\infty}(V;\RR^m)$ as $l\to +\infty$, and
\begin{equation}\label{boundlocal}
\sup_{l\ge 1}\Vert \nabla w_{l,s}\Vert_\infty\le \Vert\xi\Vert+r.
\end{equation}
Let a cube $\eps(z+ n_sY)\subset V_{l,s}$, by a change of variable and periodicity
\begin{align*}
\int_{\eps(z+ n_sY)} j\left(\frac{x}{\eps_l}, \nabla w_{l,s}(x)\right)dx&=\eps^d\int_{n_sY} j(x,\xi+\nabla\phi_s(x))dx\\
&=\vert \eps_l(z+n_sY)\vert\frac{1}{n_s^d}\int_{n_sY}  j(x,\xi+\nabla\phi_s(x))dx\\
&\le\vert \eps_l(z+n_sY)\vert \left(\H_r j(\xi)+\frac{1}{s}\right).
\end{align*}
It follows that 
\[
\int_{V_{l,s}} j\left(\frac{x}{\eps_l}, \nabla w_{l,s}(x)\right)dx\le \vert V_{l,s}\vert \left(\H_r j(\xi)+\frac{1}{s}\right).
\]
Then 
\[
\int_{V} j\left(\frac{x}{\eps_l}, \nabla w_{l,s}(x)\right)dx\le \vert V_{l,s}\vert \left(\H_r j(\xi)+\frac{1}{s}\right)+\int_{V\setminus V_{l,s}}j\left(\frac{x}{\eps_l},\xi\right)dx.
\]
Using periodicity and $\lim_{l\to +\infty}\vert V\setminus V_{l,s}\vert=0$, we obtain
\[
\limsup_{l\to +\infty}\int_{V} j\left(\frac{x}{\eps_l}, \nabla w_{l,s}(x)\right)dx\le\vert V\vert\left( \H_r j(\xi)+\frac{1}{s}\right).
\]
Letting $s\to +\infty$
\begin{equation}\label{diag1}
\limsup_{s\to +\infty}\limsup_{l\to +\infty}\int_{V} j\left(\frac{x}{\eps_l}, \nabla w_{l,s}(x)\right)dx\le\vert V\vert\H_r j(\xi).
\end{equation}
By compactness imbedding there exists a subsequence (not relabeled) $\{w_{l,s}\}_{l\ge 1}\subset l_\xi+W^{1,\infty}_0(V;\RR^m)$ such that $w_{l,s}\to l_\xi$ in $L^1(\Omega;\RR^m)$. Thus
\begin{equation}\label{diag2}
\limsup_{s\to +\infty}\limsup_{l\to +\infty}\Vert w_{l,s}-l_\xi\Vert_1=0.
\end{equation}
A simultaneous diagonalization of \eqref{diag1} and \eqref{diag2} together with the uniform bound \eqref{boundlocal} give a  sequence $\{w_l\}_{l\ge 1}\subset l_\xi+W^{1,\infty}_0(V;\RR^m)$ such that 
\[\left\{
\begin{array}{ll}
\displaystyle  w_l\stackrel{\ast}{\wto} l_\xi \mbox{ in }  W^{1,\infty}(V;\RR^m) \quad\mbox{ and }\quad \sup_{l\ge 1}\Vert \nabla w_{l}\Vert_\infty\le \Vert\xi\Vert+r,\\
 \displaystyle \limsup_{l\to +\infty}\int_{V} j\left(\frac{x}{\eps_l}, \nabla w_{l}(x)\right)dx\le\vert V\vert\H_r j(\xi). \\   
\end{array}
\right.
\]
Now, let $w\in\Aff(U;\RR^m)$ be such that $\nabla w(x)\in{\mathfrak D}$ for almost all $x\in U$. There exists a finite set $I$ such that $\nabla w(x)=\xi_i$ a.e. in $V_i$ where $\xi_i\in \MM^{m\times d}$, $V_i\subset U$ is open and $i\in I$. Therefore
\[
\int_U\H_r j(\nabla w(x))dx=\sum_{i\in I}\vert V_i\vert\H_r j(\xi_i).
\]
For every $i\in I$ there exists $\{w_\eps^i\}_{\eps>0}\subset l_{\xi_i}+W_0^{1,\infty}(V_i;\RR^m)$ such that 
\begin{equation}\label{locallimsup}
\left\{
\begin{array}{ll}
w_\eps^i\stackrel{\ast}{\wto} l_{\xi_i}\mbox{ in }W^{1,\infty}(V_i;\RR^m)\quad\mbox{ and }\quad\sup_{\eps>0}\Vert\nabla w_\eps^i\Vert_\infty\le\Vert\xi_i\Vert+r,\\
\displaystyle\limsup_{\eps\to 0}\int_{V_i} j\left(\frac{x}{\eps}, \nabla w_\eps^i(x)\right)dx\le \vert V_i\vert\H_r j(\xi_i)\\
\end{array}
\right.
\end{equation}
For any $\eps>0$, set $v_\eps=w_\eps^i-l_{\xi_i}$ on $V_i$ for all $i\in I$, then $v_\eps\in W_0^{1,\infty}(U;\RR^m)$ and $v_\eps\stackrel{\ast}{\wto} 0\mbox{ in }W^{1,\infty}(U;\RR^m)$. Define $w_\eps=v_\eps+w\in w+W^{1,\infty}_0(U;\RR^m)$ for all $\eps>0$ then $w_\eps\stackrel{\ast}{\wto} w\mbox{ in }W^{1,\infty}(U;\RR^m)$, $\sup_{\eps>0}\Vert\nabla w_\eps\Vert_\infty\le \Vert\nabla w\Vert_\infty+r$, and using \eqref{locallimsup}
\begin{align*}
\limsup_{\eps\to 0}\int_{U} j\left(\frac{x}{\eps}, \nabla w_\eps(x)\right)dx&=\limsup_{\eps\to 0}\int_{U} j\left(\frac{x}{\eps}, \nabla v_\eps(x)+\nabla w(x)\right)dx\\
&=\limsup_{\eps\to 0}\sum_{i\in I}\int_{V_i} j\left(\frac{x}{\eps}, \nabla w_\eps^i(x)\right)dx\\
&\le\sum_{i\in I}\limsup_{\eps\to 0}\int_{V_i} j\left(\frac{x}{\eps}, \nabla w_\eps^i(x)\right)dx\\
&\le\sum_{i\in I}\vert V_i\vert\H_r j(\xi_i)=\int_U \H_r j(\nabla w(x))dx.
\end{align*}

Let us prove the second part of the lemma. Let $w\in\Aff(\Omega;\RR^m)$ be such that $\int_U \H j(\nabla w(x))dx<+\infty$. Since $\dom j(x,\cdot)\subset B_R(0)$ for all $x\in\RR^d$, we have $\H_R j=\H j$.
Let $\{w_{\eps}\}_{\eps>0}\subset w+W^{1,\infty}_0(U;\RR^m)$ be such that $w_{\eps}\stackrel{\ast}{\wto} w\mbox{ in }W^{1,\infty}(U;\RR^m)$ as $\eps\to 0$, $\sup_{\eps>0}\Vert \nabla w_{\eps}\Vert_\infty\le R+\Vert\nabla w\Vert_\infty $ and 
\[
\limsup_{\eps\to 0}\int_{U} j\left(\frac{x}{\eps}, \nabla w_\eps(x)\right)dx\le \int_{U} \H j\left(\nabla w(x)\right)dx.
\]
\begin{sloppypar}
\noindent By compactness imbedding theorem in $L^1(U;\RR^m)$, there exists a subsequence $\{w_{\eps}\}_{\eps>0}\subset w+W^{1,\infty}_0(U;\RR^m)$ such that $w_{\eps}\to w\mbox{ in }L^1(U;\RR^m)$.\end{sloppypar}
\end{proof}

{\noindent{\em Proof of quasiconvexity of $\H f$.}} It is easy to see that $r\mapsto \H_r j(\cdot)$ is nonincreasing and for every $\xi\in\MM^{m\times d}$
\begin{equation}
\inf_{r>0}\H_rj(\xi)=\lim_{r\to +\infty}\H_rj(\xi)=\H j(\xi).
\end{equation}
Let $\xi\in\MM^{m\times d}$ and $\phi\in\Aff_0(Y;\RR^m)$. Denote again by $\phi$ the $Y$-periodic extension to $\RR^d$ of $\phi$. Let $\{u_n\}_{n\ge 1}\subset \Aff(\Omega;\RR^m)$ be defined by $u_n(x)=\xi x+\frac{1}{n}\phi(nx)$ for all $x\in\Omega$. By periodicity, $u_n\stackrel{\ast}{\wto} l_\xi$ in $W^{1,\infty}(\Omega;\RR^m)$. Let $r>0$. For any finite $r$,  $G_{r}$ is sequentially weak$^\ast$ lower semicontinuous in $W^{1,\infty}(\Omega;\RR^m)$. Using periodicity and the fact that $\sup_{n\ge 1}\Vert\nabla u_n\Vert_\infty\le\Vert\xi\Vert+\Vert\nabla\phi\Vert_\infty$, we deduce 
\begin{align*}
\vert\Omega\vert \H f(\xi)\le G_s(l_\xi)\le \liminf_{n\to +\infty}G_s(u_n)&\le \liminf_{n\to +\infty}H_{r+\Vert\nabla u_n\Vert_\infty}(u_n)\\&\le\lim_{n\to +\infty}\int_\Omega \H_r f\left(\xi+\nabla \phi\left(nx\right)\right)dx\\&=\vert\Omega\vert \int_Y \H_r f(\xi+\nabla \phi(x))dx.
\end{align*}
where $s=r+\Vert\xi\Vert+\Vert\nabla \phi\Vert_\infty$. Letting $r\to +\infty$ we obtain
\[
\H f(\xi)\le \int_Y \H f(\xi+\nabla \phi(x))dx,
\]
hence $\Z\H f=\H f$. Since $\H f$ is finite and Lemma \ref{quasifinite}, it follows that $\Q\H f=\H f$. \hfill $\square$
\section{Proof of Theorem \ref{main result 1} and \ref{main result 1bis}}
\subsection{Proof of the $\Gamma\mbox{-}\liminf_{\eps\to 0}I_\eps$ and $\Gamma\mbox{-}\liminf_{\eps\to 0}J_\eps$}\label{gliminf}
We only give the proof of the lower bound for $I_\eps$, the same proof works for  $J_\eps$ with the necessary changes.  

It suffices to show that for every $u\in W^{1,\infty}(\Omega;\RR^m)$
\[
\Gamma\mbox{-}\liminf_{\eps\to 0} I_{\eps}(u)\ge \int_\Omega \overline{\H\h}(\nabla u(x))dx.
\]
Without loss of generality, we can consider~$u,\{u_\eps\}_{\eps>0}\subset W^{1,\infty}(\Omega;\RR^m)$ satisfying $u_\eps\to u$ in $L^1(\Omega;\RR^m)$ and
\[
\sup_{\eps>0}I_\eps(u_\eps)<+\infty\;\mbox{ and }\;\lim_{\eps\to 0}I_\eps(u_\eps)=\liminf_{\eps\to 0}I_\eps(u_\eps).
\]
Since $\dom\h(x,\cdot)\subset \overline{\C}$ for all $x\in\RR^d$, the sequence $\{\nabla u_\eps\}_{\eps>0}\subset L^\infty(\Omega;\MM^{m\times d})$ is bounded. By Poincar\'e-Wirtinger inequality, for some $C>0$, we have
\[
\sup_{\eps>0}\Vert u_\eps\Vert_{\infty,\Omega}\le C\left\{\frac{1}{\vert \Omega\vert}\sup_{\eps>0}\Vert u_\eps\Vert_{1,\Omega}+\sup_{\eps>0}\Vert \nabla u_\eps\Vert_{\infty,\Omega}\right\}.
\]
Since $u_\eps\to u$ in $L^1(\Omega;\RR^m)$, we have $\sup_{\eps>0}\Vert u_\eps\Vert_{1,\Omega}<+\infty$, it follows that $\{u_\eps\}_{\eps>0}\subset W^{1,\infty}(\Omega;\RR^m)$ is bounded. Thus, there exists a subsequence (not relabeled) such that $u_\eps\stackrel{\ast}{\wto} u$ in $W^{1,\infty}(\Omega;\RR^m)$. Using Theorem~\ref{main result 2} \ref{mainresult22} and Theorem~\ref{BraidesTheorem}~\eqref{braides1} with $r=+\infty$
\begin{align*}
\liminf_{\eps\to 0}I_\eps(u_\eps)\ge \liminf_{\eps\to 0}\int_\Omega \h_n\left(\frac{x}{\eps},\nabla u_\eps(x)\right)dx\ge\int_\Omega \H\h_n\left(\nabla u(x)\right)dx,\\
\end{align*}
for all $n\in\NN$.
We finish the proof by using monotone convergence theorem and Theorem~\ref{main result 2} \ref{mainresult23}.
\hfill$\square$
\begin{remark}\label{notconnected} For the proof of $\Gamma\mbox{-}\liminf_{\eps\to 0}J_\eps$ instead of Poincar\'e-Wirtinger inequality we can use Poincar\'e inequality since the zero boundary conditions (note that in this case it is not necessary to assume that $\Omega$ is connected).
\end{remark}
\subsection{Proof of the $\Gamma\mbox{-}\limsup_{\eps\to 0}I_\eps$ and $\Gamma\mbox{-}\limsup_{\eps\to 0}J_\eps$}\label{glimsup} By an easy adaptation of Lemma 3.1 in \cite{oah-jpm08a} (see also \cite{oah-jpm07}, \cite{oah-jpm06}) we have
\begin{lemma}\label{relaxation aff gene} Let $f:\MM^{m\times d}\to [0,+\infty]$ be a Borel measurable function and $U\subset\RR^d$ be a bounded open set. Then for every $u\in\Aff(U;\RR^m)$ 
\[
\inf\left\{\limsup_{n\to +\infty} \int_U f(\nabla u_n)dx : u+\Aff_0(U;\RR^m)\ni u_n{\swto} u \hbox{ in }W^{1,\infty}\right\}\le \int_U \Z f(\nabla u)dx.
\]
\end{lemma}

The following approximation result is needed for the proof of the $\Gamma\mbox{-}\limsup$, see for instance \cite[Theorem 10.16 and Corollary 10.21]{dacorogna-marcellini99} or \cite[Proposition 5.1.]{oah10} for a proof.
\begin{lemma}\label{approximation of sobolev functions} Let $K\subset \MM^{m\times d}$ be a compact convex set with $0\in\inte K$. Let $t\in [0,1[$ and $u\in W^{1,\infty}(\Omega;\RR^m)$ be such that
\[
\nabla u(x)\in tK \hbox{ a.e. in }\Omega.
\]
Then for every integer $n> \frac{1}{1-t}$ there exist $u_n\in W^{1,\infty}(\Omega;\RR^m)$ and an open set  $\Omega_n\subset~\Omega$ such that
\begin{enumerate}[label={\rm (\roman{*})}, ref={\rm (\roman{*})}]
\item $\displaystyle u_n\lfloor_{\Omega_n}\in\Aff(\Omega_n;\RR^m)$ and $u_n=u$ on $\partial\Omega$;
\item $\displaystyle\|u_n- u\|_{1,p,\Omega}\le {n^{-1}} \hbox{ for all }p\in [1,+\infty[$;
\item $\displaystyle \nabla u_n(x)\in \left(t+\frac{1}{n}\right)K\hbox{ a.e. in }\Omega$;
\item $\displaystyle\left\vert\Omega\setminus\Omega_n\right\vert\le {1\over n} \hbox{ and }\vert\partial\Omega_n\vert=0$.
\end{enumerate}
\end{lemma}

Now, we are able to extend Lemma \ref{homog lemma} to Lipschitz functions with gradients compactly included in $\C$ by using Lemma \ref{approximation of sobolev functions}.
\begin{proposition}\label{upper bound prop} Assume that {\rm (H$_1$)} and {\rm (H$_2$)} hold. Let $u\in W^{1,\infty}(\Omega;\RR^m)$ and $t\in [0,1[$ such that
\[
\nabla u(x)\in t\overline{\C}\hbox{ a.e. in }\Omega.
\]
Let $\{\eps_s\}_{s\in\NN^*}$ be such that $\lim_{s\to +\infty}\eps_s=0$. Then there exist sequences $\{\tau_s\}_{s\in\NN^*}\subset [0,1[$ and $\{u_s\}_{s\in\NN^*}\subset \tau_su+W_0^{1,\infty}(\Omega;\RR^m)$ such that $[0,1[\ni\tau_s\to1$ as $s\to +\infty$,
\[
u_s\to u \hbox{ in }L^{1} \quad\hbox{ and }\quad\limsup_{s\to +\infty}\int_\Omega \h\left(\frac{x}{\eps_s},\nabla u_s(x)\right)dx\le \int_\Omega \Z\H\h(\nabla u(x))dx.
\]
\end{proposition}
\begin{proof} Let $\{\eps_s\}_{s\in\NN}\subset ]0,+\infty[$ be such that $\lim_{s\to +\infty}\eps_s=0$. Let $u\in W^{1,\infty}(\Omega;\RR^m)$ and $t\in [0,1[$ such that
$
\nabla u(x)\in t\overline{\C}\hbox{ a.e. in }\Omega.
$
By Proposition \ref{approximation of sobolev functions} there exist $\{u_n\}_{n\in\NN^*}\subset W^{1,\infty}(\Omega;\RR^m)$ and a sequence of open sets  $\{\Omega_n\}_{n\in\NN^*},\Omega_n\subset\Omega$ such that for every integer $n> (1-t)^{-1}$
\begin{enumerate}[label=$\diamond$, ref=$\diamond$]
\item $\displaystyle u_n\lfloor_{\Omega_n}\in\Aff(\Omega_n;\RR^m)$ and $u_n=u$ on $\partial\Omega$;
\item $\displaystyle\Vert u_n-u\Vert_{1,p,\Omega}\le {n^{-1}}\hbox{ for all }p\in [1,+\infty[$;
\item $\displaystyle \nabla u_n(x)\in (t+n^{-1})\overline{\C}\hbox{ a.e. in }\Omega$;
\item $\displaystyle\left\vert\Omega\setminus\Omega_n\right\vert\le {n^{-1}} \hbox{ and }\vert\partial\Omega_n\vert=0$.
\end{enumerate}
We assume, up to a subsequence, that $\nabla u_n(\cdot)\to\nabla u(\cdot)$ a.e. in $\Omega$ and $u_n\to u$ in $L^1$ as $n\to +\infty$. Choose $n_t\in\NN^*$ in order to have 
\[
t+\frac{1}{n_t}<\frac{1+t}{2}<1.
\]
Let $n\ge n_t$, then $\nabla u_n(x)\in \frac{1+t}{2} \overline{\C}$ a.e. in $\Omega$. Set $M_t=\sup\{\Z\H\h(\xi):\xi\in \frac{1+t}{2} \overline{\C}\}$ which is finite since Lemma \ref{decomposition} \ref{decomposition4}. Since $\nabla u_n(x),\nabla u(x)\in \frac{1+t}{2}\overline{\C}$ a.e. in $\Omega$ for all $n\ge n_t$, using continuity of $\Z\H\h\lfloor_{\inte\C}$ (combine Lemma \ref{fonseca} and Lemma \ref{decomposition} \ref{decomposition6}), and Lebesgue dominated convergence theorem, we obtain
\begin{equation}\label{strong fin}
\limsup_{n\to +\infty}\int_\Omega \Z\H\h(\nabla u_n(x))dx\le \int_\Omega \Z\H\h(\nabla u(x))dx.
\end{equation}
By Lemma \ref{relaxation aff gene}, for every $n\ge n_t$ there exists $\{v_{m,n}\}_{m\in\NN^*}\subset u_{n}+\Aff_0(\Omega_n;\RR^m)$ such that $v_{m,n}\swto u_{n}$ in $W^{1,\infty}$ as $m\to +\infty$, and
\begin{equation}\label{eq1 upper bound prop}
\limsup_{m\to +\infty}\int_{\Omega_n} \H\h(\nabla v_{m,n}(x))dx\le \int_{\Omega_n} \Z\H\h(\nabla u_{n}(x))dx\le M_t\vert\Omega\vert.
\end{equation}
By compact imbedding theorem for every $n \ge n_t$ there exists a subsequence (not relabeled) $v_{m,n}\to u_{n}$ in $L^1(\Omega_n)$ when $m\to +\infty$, and by \eqref{eq1 upper bound prop} we can assume that $\nabla v_{m,n}(\cdot)\in \dom\H\h$. 

Let $n\ge n_t$, $m\in \NN$ and $\tau\in [0,1[$. We have $\tau v_{m,n}\in \tau u_{n}+\Aff_0(\Omega_n;\RR^m)$ and $\tau\nabla v_{m,n}(\cdot)\in \tau\overline{\C}\subset\inte\C$ a.e. in $\Omega_n$. By Lemma \ref{homog lemma}, there exists a sequence $\{u_{s,m,n}^\tau\}_{s\in\NN}\subset \tau v_{m,n}+W^{1,\infty}_0(\Omega_n;\RR^m)$ such that $u_{s,m,n}^\tau\to \tau v_{m,n}$ in $L^{1}(\Omega_n;\RR^m)$  when $s\to +\infty$, and 
\[
\limsup_{s\to +\infty}\int_{\Omega_n}\h\left(\frac{x}{\eps_s},\nabla u_{s,m,n}^\tau(x)\right)dx\le \int_{\Omega_n}\H\h(\tau\nabla v_{m,n}(x))dx.
\]
Let $\{\tau_l\}_{l\in\NN^*}\subset [0,1[$ be such that $\lim_{l\to +\infty}\tau_l=1$. By Corollary \ref{exist lim} \ref{ruuscHp}, for every $l\ge 1$
\[
\int_{\Omega_n} \H\h(\tau_l \nabla v_{m,n}(x))dx\le \frac1l+\int_{\Omega_n} \H\h(\nabla v_{m,n}(x))dx.
\]
Fix $l\ge 1$. We deduce that there exists $\{u_{s,m,n}^{\tau_l}\}_{s\in\NN}\subset \tau_l v_{m,n}+W^{1,\infty}_0(\Omega_n;\RR^m)$ such that $u_{s,m,n}^{\tau_l}\to \tau_l v_{m,n}$ in $L^{1}(\Omega_n;\RR^m)$  when $s\to +\infty$, and 
\[
\limsup_{s\to +\infty}\int_{\Omega_n}\h\left(\frac{x}{\eps_s},\nabla u_{s,m,n}^{\tau_l}(x)\right)dx\le \frac1l+\int_{\Omega_n} \H\h(\nabla v_{m,n}(x))dx.
\]
Set $u_{s,m,n,l}={\mathbb I}_{\Omega_n}u_{s,m,n}^{\tau_l}+{\mathbb I}_{\Omega\setminus\Omega_n}\tau_l u_{n}\in \tau_lu+W^{1,\infty}_0(\Omega;\RR^m)$, where \[{\mathbb I}_S(x)=\left\{\begin{array}{ll}1&\mbox{ if }x\in S\\0&\mbox{ if }x\notin S\end{array}\right.\] for any $S\subset\RR^d$. We have
\begin{align*}
\label{eq1}\limsup_{s\to +\infty}\left\{\int_{\Omega} \h\left(\frac{x}{\eps_s},\nabla u_{s,m,n,l}(x)\right)dx\right.&-\left.\int_{\Omega\setminus\Omega_n}\h\left(\frac{x}{\eps_s},\nabla u_{s,m,n,l}(x)\right)dx\right\}\\&\le \frac1l+\int_{\Omega_n} \H\h(\nabla v_{m,n}(x))dx.\nonumber
\end{align*}
Set $M_{s,l}=\int_\Omega\sup\{\h(\frac{x}{\eps_s},\xi):\xi\in \tau_l \overline{\C}\}dx$, by {\rm (H$_2$)} and periodicity
\[
M_l=\lim_{s\to +\infty}M_{s,l}=\vert \Omega\vert \int_Y \sup\{\h(x,\xi):\xi\in \tau_l \overline{\C}\}dx.
\]
We have
\[
\left\vert\int_{\Omega\setminus\Omega_n}\h\left(\frac{x}{\eps_s},\nabla u_{s,m,n,l}(x)\right)dx\right\vert\le M_{s,l}\vert\Omega\setminus\Omega_n\vert=\frac{1}{n} M_{s,l}.
\]
We deduce 
\[
\limsup_{m\to +\infty}\limsup_{s\to +\infty}\int_{\Omega} \h\left(\frac{x}{\eps_s},\nabla u_{s,m,n,l}(x)\right)dx\le
\frac{1}{n} M_{l}+\frac1l+\int_{\Omega_n} \Z\H\h(\nabla u_{n}(x))dx.
\]
Using \eqref{strong fin}, we obtain
\begin{equation}\label{limsup1}
\limsup_{l\to +\infty}\limsup_{n\to +\infty}\limsup_{m\to +\infty}\limsup_{s\to +\infty}\int_{\Omega} \h\left(\frac{x}{\eps_s},\nabla u_{s,m,n,l}(x)\right)dx\le\int_{\Omega} \Z\H\h(\nabla u(x))dx.
\end{equation}
We also have for every $s\ge 1$, $m\ge 1$, $n\ge n_t$ and $l\ge 1$
\begin{align*}
\int_{\Omega}\vert u_{s,m,n,l}(x)-u(x)\vert dx&\le \int_{\Omega}\vert u_{s,m,n,l}(x)-u_{m,n,l}(x)\vert dx\\
&+\int_{\Omega}\vert u_{m,n,l}(x)-\tau_lu_{n}(x)\vert dx+\int_{\Omega}\vert \tau_lu_{n}(x)-\tau_l u(x)\vert dx\\
&+\int_{\Omega}\vert \tau_lu(x)-u(x)\vert dx,
\end{align*}
where $u_{m,n,l}=\tau_l\left({\mathbb I}_{\Omega_n} v_{m,n}+{\mathbb I}_{\Omega\setminus\Omega_n}u_n\right)$. Letting $s\to +\infty$, $m\to +\infty$, $n\to +\infty$ and $l\to +\infty$, we obtain 
\begin{equation}\label{limsup2}
\limsup_{l\to +\infty}\limsup_{n\to +\infty}\limsup_{m\to +\infty}\limsup_{s\to +\infty}\Vert u_{s,m,n,l}-u\Vert_{1,\Omega}=0.
\end{equation}
Diagonalization arguments applying to \eqref{limsup1} and \eqref{limsup2} give a sequence $\{u_s\}_{s\ge 1}\subset \tau_{l(s)} u+W^{1,\infty}_0(\Omega;\RR^m)$ with $s\mapsto\tau_{l(s)}$ a subsequence of $\{\tau_l\}_{l\in\NN^*}$ such that 
\[
\left\{\begin{array}{ll}
\displaystyle\limsup_{s\to +\infty} \int_\Omega \h\left(\frac{x}{\eps_s},\nabla u_s(x)\right)dx\le\int_{\Omega} \Z\H\h(\nabla u(x))dx\\ \\
\tau_{l(s)} u+W^{1,\infty}_0(\Omega;\RR^m)\ni u_s\to u\hbox{ in }L^{1}\hbox{ as }s\to +\infty.
\end{array}
\right.
\]
\end{proof}
\subsubsection{Proof of the $\Gamma\mbox{-}\limsup_{\eps\to 0}I_\eps$ and $\Gamma\mbox{-}\limsup_{\eps\to 0}J_\eps$} The proof of the upper bound follows by using Proposition~\ref{upper bound prop}, the following lemma, and Theorem \ref{main result 2}. For simplicity the following lemma is stated for $I_\eps$, however it is working for $J_\eps$ with the necessary changes.
\begin{lemma}\label{strong envelope} Assume that {\rm (H$_1$)} and {\rm (H$_2$)} hold. If for every $u\in W^{1,\infty}(\Omega;\RR^m)$ and $t\in [0,1[$~such that $\nabla u(x)\in t\overline{\C}$ a.e. in $\Omega$ it holds
\[
\Gamma\mbox{-}\limsup_{\eps\to 0} I_{\eps}(u)\le \int_\Omega \Z\H\h(\nabla u(x))dx.
\]
Then for every $u\in W^{1,\infty}(\Omega;\RR^m)$
\[
\Gamma\mbox{-}\limsup_{\eps\to 0} I_{\eps}(u)\le \int_\Omega\widehat{\Z\H\h}(\nabla u(x))dx.
\]
\end{lemma}
\begin{proof} Without loss of generality, assume that
\[
\int_\Omega \widehat{\Z\H\h}(\nabla u(x))dx<+\infty
\]
with $u\in W^{1,\infty}(\Omega;\RR^m)$. We deduce that $\nabla u(x)\in \overline{\C}$ a.e. in $\Omega$ by Corollary \ref{exist lim}. Let $n\ge 1$. Set $u_n=(1-{1\over n})u\in W^{1,\infty}(\Omega;\RR^m)$, then $u_n\to u$ in $L^{1}$. By $\lsp$, it holds that $\nabla u_n(x)\in (1-{1\over n})\overline{\C}\subset\inte\C$ a.e. in $\Omega$. Let $\eps>0$. By Corollary \ref{exist lim}, there exists $\eta>0$ such that for every $\xi\in\overline{\C}$ and $t\in [0,1[$, if $1-t<\eta$ then $\Z\H\h(t\xi)\le \widehat{\Z\H\h}(\xi)+\eps$. Choose $n_\eps\in\NN^*$ such that $1-{1\over n}<\eta$ for all $n\ge n_\eps$. Thus, for every $n\ge n_\eps$
\[
\Gamma\mbox{-}\limsup_{\eps\to 0} I_{\eps}(u_n)\le \int_\Omega \Z\H\h(\nabla u_n(x))dx\le \eps\vert\Omega\vert+\int_\Omega \widehat{\Z\H\h}(\nabla u(x))dx.
\]
Letting $n\to +\infty$ and using the $L^1$ sequential lower semicontinuity of $\displaystyle\Gamma\mbox{-}\limsup_{\eps\to 0} I_{\eps}$, we deduce
\[
\Gamma\mbox{-}\limsup_{\eps\to 0} I_{\eps}(u)\le \eps\vert\Omega\vert+\int_\Omega \widehat{\Z\H\h}(\nabla u(x))dx.
\]
Since $\eps>0$ is arbitrary, we obtain the desired result.
\end{proof}
\section{Consequences of line segment principle and proof of Theorem \ref{main result 2}}\label{lineseg}
\subsection{Some properties of $\H\h$ and $\Z\H\h$}
\subsubsection{Some consequences of line segment principle}
\begin{lemma}\label{decomposition} We have 
\begin{enumerate}[label={\rm (\roman{*})}, ref={\rm (\roman{*})}]
\item For every increasing sequence $\{t_n\}_{n\in\NN^*}\subset [0,1[$ satisfying $\lim_{n\to +\infty} t_n=1$, it holds
\[
\inte\C={\bigcup_{t\in [0,1[}t\overline{\C}}={\bigcup_{n\in\NN^\star}t_n\overline{\C}},\hbox{ and }\inte\C={\bigcup_{t\in [0,1[}t\,\inte{\C}}={\bigcup_{n\in\NN^\star}t_n\inte{\C}}.
\]\label{decomposition1}
\item If $K\subset\inte\C$ is compact then $K\subset t\,\inte{\C}$ for some $t\in [0,1[$\label{decomposition2}.
\item The function $\h$ is locally bounded in $\inte\C$ if and only if 
\[
\sup\{\h(\cdot,\xi):\xi\in t\overline{\C}\}\in L^\infty_{\rm loc}(\RR^d)
\]
 for all $t\in [0,1[$ \label{decomposition3}.
\item Assume that {\rm (H$_2$)} holds. For every $t\in [0,1[$, it holds 
\[
\sup\{\H\h(\xi):\xi\in t\overline{\C}\}<+\infty.
\]\label{decomposition4}
\item Assume that $\h$ is locally bounded in $\inte\C$. Then {\rm(}${\rm H}_3${\rm )} holds if and only if there exists an increasing sequence such that $[0,1[\ni t_n\to 1$, and for every $n\ge 1$ and every $x\in \RR^d$ 
\[
\inf_{\xi\in \overline{\C}\setminus t_n\overline{\C}} \h(x,\xi)\ge n.
\]\label{decomposition5}
\item \label{decomposition6}  Assume that {\rm (H$_2$)} holds. Then 
\[
\inte\C\subset\dom\H\h\subset\dom\overline{\H\h}\subset\overline{\C}\;\mbox{ and }\;
\inte\C\subset\dom\H\h\subset\dom\Z\H\h\subset\overline{\C},
\]
\[
\inte(\dom\H\h)=\inte(\dom\overline{\H\h})=\inte(\dom\Z\H\h)=\inte(\dom\h(x,\cdot))=\inte\C\mbox{ a.e. in }Y.
\]
\end{enumerate}
\end{lemma}
\begin{proof}
Items \ref{decomposition1} and \ref{decomposition2} are direct consequences of $\lsp$, see \cite{oah10} for a proof. Item \ref{decomposition3} is a consequence of \ref{decomposition1}, \ref{decomposition2} and (H$_2$). Item \ref{decomposition4} follows by using \ref{decomposition3}
\begin{equation}\label{unif}
\sup_{\xi\in t\overline{\C}}\H\h(\xi)\le \int_Y\sup_{\xi\in t\overline{\C}}\h(x,\xi)dx<+\infty.
\end{equation}

The proof of item \ref{decomposition5} is given following the lines of the proof of Lemma 2.1. (iv) in \cite{oah10}. By {\rm (H$_3$)}, we can find a sequence of compact set $\{K_n\}_{n\in\NN^*}\subset \inte\C$ such that for every $x\in\RR^d$
\[
K_1\subset K_2\subset \cdots\subset K_n\subset\cdots\subset\bigcup_{n\in\NN^*}K_n=K_\infty,\;\hbox{ and } \inf_{\overline{\C}\setminus K_n}\h(x,\cdot)\ge n \hbox{ for all }n\ge 1.
\]
Thus $\inf_{\overline{\C}\setminus K_\infty}\h(x,\cdot)=+\infty$ for all $x\in\RR^d$. Assume that $K_\infty\not=\inte\C$, then there exists $\xi_0\in\inte\C\setminus K_\infty$ such that $\h(x,\xi_0)=+\infty$ for all $x\in\RR^d$, which is impossible since {\rm (H$_2$)}. Thus $K_\infty=\inte\C$. By \ref{decomposition2}, we can build an increasing sequence $\{t_n\}_{n\in\NN^*}\subset [0,1[$ such that $K_n\subset t_n\overline{\C}$ for all $n\in\NN^*$. It follows that $\inte\C=K_\infty=\cup_{n\ge 1}t_n\overline{\C}$ and therefore the sequence $t_n\to 1$ as $n\to +\infty$, indeed we cannot have $\tau=\sup_{n\ge 1}t_n=\lim_{n\to +\infty}t_n<1$, otherwise, by $\lsp$ $\inte\C\subset\tau\inte\C$ which is impossible since $\inte\C\not=\emptyset$. We also have for every $n\in\NN^*$ and every $x\in\RR^d$
\[
\inf_{\xi\in\overline{\C}\setminus t_n\overline{\C}} \h(x,\xi)\ge\inf_{\xi\in\overline{C}\setminus K_n}\h(x,\xi)\ge n.
\]
The other implication is easier. Let $s>0$. Let $n\ge s$ and choose $K_s=t_n\overline{\C}$ then $\inf_{\overline{\C}\setminus K_s}\h(x,\cdot)\ge n\ge s$ for all $x\in\RR^d$.

Let us prove \ref{decomposition6}. Let $\xi\in\inte\C$  then there exists $t\in [0,1[$ such that $\xi\in t\overline{\C}$. By definition of $\H\h$ and \eqref{unif}
\[
\overline{\H\h}(\xi)\le\H\h(\xi)<+\infty\;\mbox{ and }\;\Z{\H\h}(\xi)\le\H\h(\xi)<+\infty.
\]
We deduce 
\[
\inte\C\subset\dom\H\h\subset\dom\overline{\H\h}\;\mbox{ and }\;\inte\C\subset\dom{\H\h}\subset\dom\Z\H\h.
\]
It remains to prove that $\dom\overline{\H\h}\subset\overline{\C}$ and $\dom\Z\H\h\subset\overline{\C}$. For each $n\in\NN$, consider the function $d_n:\MM^{m\times d}\to [0,+\infty[$ defined by $ d_n(\xi)=n\dist(\xi,\overline{\C})$. It is easy to see that $d_n$ is lower semicontinuous, convex and then quasiconvex by Jensen inequality, and $d_n\le \h$ for all $n\in\NN$. Thus $d_n\le \overline{\H\h}$ and $d_n\le \Z\H\h$ for all $n\in\NN$, and the inclusions $\dom\overline{\H\h}\subset\overline{\C}$ and $\dom\Z\H\h\subset\overline{\C}$
follow since
\[
\sup_{n\in\NN}d_n(\xi)=
\left\{
\begin{array}{cl}
 0 &\hbox{ if }\xi\in\overline{\C},   \\
  +\infty&\hbox{ otherwise.}
\end{array}
\right.
\]
The third sequence of equalities follows by applying $\lsp$. The proof is finished.
\end{proof}


\subsubsection{Extension of radially uniformly upper semicontinuous functions}
\begin{definition} Let $U\subset\RR^d$ be a measurable set. We say that $f:U\times \MM^{m\times d}\to [0,+\infty]$ is {\em radially uniformly upper semicontinuous (r.u.u.s.c.)} in $A\subset\MM^{m\times d}$ if  
\begin{equation}\label{ruusc}
\forall \eps>0\;\exists \eta>0 \;\forall (x,\xi)\in U\times A\; \forall t\in [0,1[\quad
1-t<\eta\Longrightarrow f(x,t\xi)\le f(x,\xi)+\eps.
\end{equation}
If $f$ does not depend on $x$, it is {\em r.u.u.s.c.} in $A$ if $x$ is removed in \eqref{ruusc}.
\end{definition}
Here is an extension result for r.u.u.s.c. functions.
\begin{proposition}\label{extension} Let $D\subset\MM^{m\times d}$ be a set and let $f:\MM^{m\times d}\to [0,+\infty]$ be r.u.u.s.c. in $D\subset\MM^{m\times d}$. Assume that 
\begin{enumerate}[label={\rm(\roman{*})}, ref={\rm(\roman{*})}]
\item $t \overline{D}\subset D$ for all $t\in [0,1[$ \label{retract};
\item $f$ is lower semicontinuous in $D$\label{fsci}.
\end{enumerate}
Then there exists $\widehat{f}$ r.u.u.s.c. in $\overline{D}$ such that $f=\widehat{f}$ in $D$. Moreover $\widehat{f}$ is given by
\[
\widehat{f}(\xi)=\left\{
\begin{array}{ll}
f(\xi)  &  \mbox{ if }\xi\in D  \\
\displaystyle\lim_{[0,1[\ni t\to 1}f(t\xi) &  \mbox{ if }\xi\in \partial D   \\
  +\infty &   \mbox{ otherwise. }
\end{array}
\right.
\]
\end{proposition}
\begin{proof} Using Lemma \ref{limit widehat}, we define $\widehat{f}:\MM^{m\times d}\to [0,+\infty]$ by
\[
\widehat{f}(\xi)=\left\{
\begin{array}{ll}
f(\xi)  &  \mbox{ if }\xi\in D  \\
\displaystyle\lim_{[0,1[\ni t\to 1}f(t\xi) &  \mbox{ if }\xi\in \partial D   \\
  +\infty &   \mbox{ otherwise. }
\end{array}
\right.
\]
We have to show that $\widehat{f}$ is r.u.u.s.c. in $\overline{D}$. Let $\{\tau_n\}_{n\in\NN}\subset [0,1[$ be such that $\lim_{n\to +\infty}\tau_n=1$. Let $\eps>0$. There exists $\eta>0$ such that for every $n\in\NN$, every $t\in [0,1[$ and $\xi\in\overline{D}$ it holds, by taking account of \ref{retract},
\[
1-t\le\eta\implies f(t\tau_n\xi)\le f(\tau_n\xi)+\eps.
\]
Since \ref{fsci} and by Lemma \ref{limit widehat}, we deduce for every $t\in [0,1[$ and $\xi\in\overline{D}$
\[
1-t\le\eta\implies f(t\xi)\le \liminf_{n\to +\infty}f(t\tau_n\xi)\le\liminf_{n\to +\infty}f(\tau_n\xi)+\eps=\widehat{f}(\xi)+\eps.
\]
The proof is complete.
\end{proof}
The following lemma is essentially due to Wagner (see \cite{wagner09}).
\begin{lemma}\label{limit widehat} Let $f:\MM^{m\times d}\to [0,+\infty]$ be r.u.u.s.c. in $D\subset\MM^{m\times d}$. If \ref{retract} in Proposition \ref{extension} holds then for every $\xi\in\partial D$
\begin{equation}\label{liminfsup}
\limsup_{[0,1[\ni t\to 1}f(t\xi)=\liminf_{[0,1[\ni t\to 1}f(t\xi)\in [0,+\infty],
\end{equation}
In this case, we will denote by $\lim_{[0,1[\ni t\to 1}f(t\xi)$ the right or the left hand term in the inequality \eqref{liminfsup}.
\end{lemma}
\begin{proof}Let $\xi\in\partial D$. Set
$\lambda=\limsup_{[0,1[\ni t\to 1}f(t\xi)$ and $\mu=\liminf_{[0,1[\ni t\to 1}f(t\xi)$. If $\mu=+\infty$ then
\[
\lambda=\mu=\lim_{[0,1[\ni t\to 1}f(t\xi)=+\infty.
\]
Assume that $\mu<+\infty$. We have two possibilities, either $\lambda=+\infty$ or $\lambda<+\infty$.

Suppose that $\lambda=+\infty$. Consider two sequences $\{t_n\}_{n\in\NN^*}, \{\tau_n\}_{n\in\NN^*}\subset [0,1[$ such that $t_n\to 1$ and $\tau_n\to 1$ as $n\to +\infty$ satisfying
\[
\lambda=\lim_{n\to +\infty} f(t_n\xi)\hbox{ and }\mu=\lim_{n\to +\infty} f(\tau_n\xi).
\]
We can find two increasing functions $\sigma,\sigma^\prime:\NN^*\to\NN^*$ such that for every $n\in\NN^*$
\[
1-{1\over n}\le t_{\sigma(n)}< \tau_{\sigma^\prime(n)}<1.
\]
Let $\eps>0$. There exists $N_0\in\NN^*$ such that for every $n\ge N_0$ it holds
\begin{equation}\label{eq1 exist limit}
f(t_{\sigma(n)}\xi)\ge 1+\eps+\mu\;\hbox{ and }\;\vert f(\tau_{\sigma^\prime(n)}\xi)-\mu\vert\le{\eps\over 2}.
\end{equation}
Since $f$ is r.u.u.s.c. in $D$, there exists $\eta>0$ such that for every $\xi\in D$ and every $t\in [0,1[$ it holds
\begin{equation}\label{eq2 exist limit}
1-t\le\eta\implies f(t\xi)\le f(\xi)+{\eps\over 2}.
\end{equation}
Choose an integer $n\ge \max\{2,N_0, \eta^{-1}\}$. Then it holds that
\begin{equation}\label{eq3 exist limit}
\tau_{\sigma^\prime(n)}>0\;\hbox{ and }\;1-{t_{\sigma(n)}\over\tau_{\sigma^\prime(n)}}\le \eta.
\end{equation}
Therefore, by (\ref{eq1 exist limit}), (\ref{eq3 exist limit}), (\ref{eq2 exist limit}) and \ref{retract} we obtain
\begin{align*}
1+\eps+\mu &\le f(t_{\sigma(n)}\xi)-f(\tau_{\sigma^\prime(n)}\xi)+f(\tau_{\sigma^\prime(n)}\xi)\\
 &= \left(f\left({t_{\sigma(n)}\over\tau_{\sigma^\prime(n)}}\tau_{\sigma^\prime(n)}\xi\right)- f(\tau_{\sigma^\prime(n)}\xi)\right)+f(\tau_{\sigma^\prime(n)}\xi)\le \eps+\mu,
\end{align*}
which is impossible. It means that if $\mu<+\infty$ then $\lambda<+\infty$. 

Now, we will show that in this case $\mu=\lambda$. Consider $\{t_n\}_{n\in\NN^*}, \{\tau_n\}_{n\in\NN^*}\subset [0,1[$ such that $t_n\to 1$ and $\tau_n\to 1$ as $n\to +\infty$ satisfying
\[
\lambda=\lim_{n\to +\infty} f(t_n\xi)\hbox{ and }\mu=\lim_{n\to +\infty} f(\tau_n\xi).
\]
As above we can find two subsequences such that for every $n\in\NN^*$
\[
1-{1\over n}\le t_{\sigma(n)}< \tau_{\sigma^\prime(n)}<1.
\]
Let $\eps>0$. There exists $N_0\in\NN^*$ such that for every $n\ge N_0$ it holds
\begin{equation}\label{eq11 exist limit}
\vert f(t_{\sigma(n)}\xi)-\lambda\vert\le{\eps\over 3}\;\hbox{ and }\;\vert f(\tau_{\sigma^\prime(n)}\xi)-\mu\vert\le{\eps\over 3}.
\end{equation}
Since $f$ is r.u.u.s.c. in $D$, there exists $\eta>0$ such that for every $\xi\in D$ and every $t\in [0,1[$ it holds
\begin{equation}\label{eq21 exist limit}
1-t\le\eta\implies f(t\xi)\le f(\xi)+{\eps\over 3}.
\end{equation}
Choose an integer $n\ge \max\{2,N_0, \eta^{-1}\}$. Then it holds that
\begin{equation}\label{eq31 exist limit}
\tau_{\sigma^\prime(n)}>0\;\hbox{ and }\;1-{t_{\sigma(n)}\over\tau_{\sigma^\prime(n)}}\le \eta.
\end{equation}
Therefore, by (\ref{eq11 exist limit}), (\ref{eq31 exist limit}), (\ref{eq21 exist limit}) and \ref{retract} we obtain
\begin{align*}
0\le \lambda-\mu &= \lambda-f(t_{\sigma(n)}\xi)+f(t_{\sigma(n)}\xi)- f(\tau_{\sigma^\prime(n)}\xi)+ f(\tau_{\sigma^\prime(n)}\xi)-\mu\\
&\le {2\eps\over 3}+f\left({t_{\sigma(n)}\over\tau_{\sigma^\prime(n)}}\tau_{\sigma^\prime(n)}\xi\right)-f(\tau_{\sigma^\prime(n)}\xi)\le\eps.
\end{align*}
The proof is complete since $\eps>0$ is arbitrary.
\end{proof}
\begin{remark}\label{remarkwide}Here are some consequences of Proposition \ref{extension}.
\begin{enumerate} [label={\rm(\arabic{*})}, ref={\rm(\arabic{*})}]
\item \label{lscwidehat} The extension $\widehat{f}$ in Proposition \ref{extension} is lower semicontinuous in $\overline{D}$. Indeed, let $\xi\in\overline{D}$ and $\{\xi_n\}_{n\in\NN}\subset\overline{D}$ be such that $\xi_n\to\xi$ as $n\to +\infty$. Assume that $\liminf_{n\to +\infty}\widehat{f}(\xi_n)<+\infty$ (otherwise there is nothing to prove). Let $\eps>0$. There exists $\eta>0$ such that for every $t\in [0,1[$ and $n\in\NN$, $f(t\xi_n)\le \widehat{f}(\xi_n)+\eps$ whenever $1-t\le \eta$. Therefore, using \ref{retract} and \ref{fsci}
\[
\liminf_{n\to +\infty}\widehat{f}(\xi_n)+\eps\ge \liminf_{n\to +\infty}f(t\xi_n)\ge f(t\xi).
\]
Letting $t\to 1$ and $\eps\to 0$, we obtain $\liminf_{n\to +\infty}\widehat{f}(\xi_n)\ge \widehat{f}(\xi)$.
\item \label{radialcontinuity} As a consequence of \ref{lscwidehat}, the extension $\widehat{f}$ in Proposition \ref{extension} is radially continuous in $\overline{D}$, i.e., for every $\xi\in\overline{D}$
\[
\lim_{[0,1[\ni t\to 1}\widehat{f}(t\xi)=\widehat{f}(\xi).
\]
\end{enumerate} 
\end{remark}
\begin{corollary}\label{exist lim} Assume that {\rm (${\rm H}_1$)} and {\rm (${\rm H}_2$)} hold. Then 
\begin{enumerate}[label={\rm(\roman{*})}, ref={\rm(\roman{*})}]
\item $\H\h$ and $\Z\H\h$ are r.u.u.s.c. in their domains \label{ruuscHp};
\item $\Z\H\h$ admits a r.u.u.s.c. extension in $\overline{\C}$ denoted by $\widehat{\Z\H\h}$ which is given by
\[
\widehat{\Z\H\h}(\xi)=
\begin{cases}
\Z\H\h(\xi)&\mbox{ if }\xi\in\inte\C\\
\displaystyle\lim_{[0,1[\ni t\to 1}\Z\H\h(t\xi)&\mbox{ if }\xi\in\partial\C\\
+\infty&\mbox{ if }\xi\notin\overline{\C}.
\end{cases}
\]\label{ruuscextensionzh}
\end{enumerate}
\end{corollary}
\begin{proof} Let us prove \ref{ruuscHp}. Fix $\eps>0$. Since {\rm (H$_1$)}, there exists $\eta>0$ such that for every $t\in [0,1[$, $x\in\RR^d$ and $\xi\in\dom\h(x,\cdot)$, such that
\begin{equation}\label{minruuscHp0}
1-t\le \eta\Longrightarrow \h(x,t\xi)\le \h(x,\xi)+\eps.
\end{equation}
Let $\xi\in \dom\H\h$. There exist $n_0\ge 1$ and $\phi_0\in W^{1,\infty}_0(n_0Y;\RR^m)$ such that
\begin{equation}\label{minruuscHp}
+\infty>\H\h(\xi)+\frac{\eps}{2}\ge \frac{1}{n_0^d}\int_{n_0Y} \h(x,\xi+\nabla \phi_0(x))dx.
\end{equation}
It follows that $\xi+\nabla\phi_0(x)\in\dom\h(x,\cdot)$ a.e. in $n_0Y$. Using \eqref{minruuscHp0}, we have 
\[
\h(x,t(\xi+\nabla\phi_0(x)))\le \h(x,\xi+\nabla\phi_0(x))+\frac{\eps}{2}\mbox{ a.e. in }n_0Y,
\]
for all $t\in [0,1[$ such that $1-t\le \eta$. 
Then 
\[
\frac{1}{n_0^d}\int_{n_0Y}\h(x,t(\xi+\nabla\phi_0(x)))dx\le \frac{1}{n_0^d}\int_{n_0Y} \h(x,\xi+\nabla \phi_0(x))dx+\frac{\eps}{2}
\]
and by \eqref{minruuscHp}, we have
\[
\H\h(t\xi)\le \frac{1}{n_0^d}\int_{n_0Y}\h(x,t(\xi+\nabla\phi_0(x)))dx\le \H\h(\xi)+\eps
\]
for all $t\in [0,1[$ such that $1-t\le \eta$. Using the same type of arguments and that $\H\h$ is r.u.u.s.c. in $\dom\H\h$, it is easy to show that $\Z\H\h$ is r.u.u.s.c. in $\dom\Z\H\h$.

By \ref{ruuscHp}, Lemma \ref{decomposition} \ref{decomposition6} and Lemma \ref{fonseca}, $\Z\H\h$ is r.u.u.s.c. in $\inte\C$ and is continuous in $\inte\C$. The proof of \ref{ruuscextensionzh} follows by applying Proposition \ref{extension}.
\end{proof}

From Corollary \ref{exist lim}, Lemma \ref{fonseca} and Theorem \ref{main result 2}, we deduce
\begin{corollary}\label{cor4} Assume that {\rm (${\rm H}_1$)}, {\rm (${\rm H}_2$)} and {\rm (${\rm H}_3$)}. If there exists a convex function $\psi:\MM^{m\times d}\to [0,+\infty]$ satisfying 
\begin{itemize}
\item[$\diamond$] $\h(x,\xi)\ge \psi(\xi)$ for all $(x,\xi)\in\RR^d\times\MM^{m\times d}$;
\item[$\diamond$] $\limsup_{[0,1[\ni t\to 1}\psi(t\xi)=+\infty$ for all $\xi\in \partial\C$.
\end{itemize}
Then $\overline{\H\h}$ is quasiconvex and continuous, and for every $\xi\in\MM^{m\times d}$
\[
\overline{\H\h}(\xi)=
\begin{cases}
\Z\H\h(\xi)&\mbox{ if }\xi\in\inte\C\\
+\infty&\mbox{ otherwise.}
\end{cases}
\]
\end{corollary}
\subsection{Proof of Theorem \ref{main result 2}}
We will divide the proof of Theorem 1.2 into two steps. Roughly, in the first step we construct a nondecreasing sequence $\{W_n\}_{n\in\NN^*}$ ($Y$-periodic with respect to its first variable and whose supremum  is equal to $W$) such that $\mathcal{H}W_n$ is of linear growth for all $n\ge 1$, and in the second step we prove that the supremum of $\mathcal{H}W_n$ is equal to $\overline{\mathcal{H}W}$.

The following lemma is a particular case of \cite[Theorem 4]{muller99} which is a generalization of a K. Zhang result \cite{zhang92} by S. M\"uller.
\begin{lemma}\label{Muller} Let $K\subset\MM^{m\times d}$ be a compact convex set. Let $\Omega\subset\RR^d$ be a bounded open set with Lipschitz boundary. Let $\{\phi_n\}_{n\in\NN}\subset W^{1,\infty}(\Omega;\RR^m)$ be a sequence satisfying
\[
\phi_n\to \phi_\infty \hbox{ in }L^1(\Omega;\RR^m)\quad\hbox{ and }\quad\int_\Omega \dist(\nabla \phi_n(x),K)dx\to 0.
\]
Then $\phi_n\wto \phi_\infty$ in $W^{1,1}$ and there exists $\{\psi_n\}_{n\in\NN}\subset W^{1,\infty}(\Omega;\RR^m)$ such that
\[
\left\{
\begin{array}{l}
\psi_n=\phi_\infty\hbox{ on }\partial\Omega  \\
\displaystyle\left\vert\left\{ x\in\Omega: \nabla \phi_n(x)\ne\nabla \psi_n(x)\right\} \right\vert\to 0\\
\displaystyle\left\Vert \dist(\nabla \psi_n,K)\right\Vert_{\infty,\Omega}\to 0.
\end{array}
\right.
\]
\end{lemma}
We will need the following lemma in the Subsection \ref{approxfin}, see \cite[Lemma 2.8]{oah10} for a proof.
\begin{lemma}\label{voisinage r} Let $r>0$. Let $\rho>0$ be such that $\rho \overline{B}\subset\inte\C$, with $\overline{B}=\{\xi\in\MM^{m\times d}:\vert\xi\vert\le 1\}$. Then
\[
\left\{\xi\in\MM^{m\times d}:\dist(\xi,\overline{\C})\le \rho\frac{r}{2}\right\}\subset (1+r)\inte\C.
\]
\end{lemma}

\subsubsection{Construction of the nondecreasing sequence}\label{construction} Consider the sequence $\{t_n\}_{n\in\NN^*}$ given by Lemma \ref{decomposition} \ref{decomposition4}. For each $n\in\NN^*$, we set 
\[
\h_n(x,\xi)=\left\{
\begin{array}{ll}
\h(x,\xi)  &   \hbox{ if }\xi\in t_n\overline{\C},\; x\in\RR^d \\
\displaystyle n\left(1+\dist(\xi,\overline{\C})\right) &   \hbox{ if }\xi\notin  t_n\overline{\C},\; x\in\RR^d.
\end{array}
\right.
\]
By (${\rm H}_2$), it holds for every $\xi\in\MM^{m\times d}$ and $n\in\NN^*$
\[
\h_n(x,\xi)\le 
\begin{cases}
\displaystyle \sup_{\xi\in t_n\overline{\C}} \h(x,\xi)&\mbox{ if }\xi\in  t_n\overline{\C},\; x\in\RR^d\\
\displaystyle n(1+\diam(\overline{\C}))(1+\vert\xi\vert)&   \mbox{ if }\xi\notin  t_n\overline{\C},\; x\in\RR^d.
\end{cases}
\]
Set $\alpha_n=\max\left\{\Vert\sup_{\xi\in t_n\overline{\C}}\h(\cdot,\xi)\Vert_{\infty}, n(1+\diam(\overline{\C}))\right\}$ where $\diam(\overline{\C})=\sup\{\vert \xi-\zeta\vert:\xi,\zeta\in\overline{\C}\}$. It follows that $\h_n(x,\xi)\le \alpha_n (1+\vert\xi\vert)$ a.e. in $\RR^d$. 
Therefore, for every $\xi\in\MM^{m\times d}$ and every $n\in\NN^*$
\[
\H\h_n(\xi)\le \alpha_n(1+\vert\xi\vert).
\]
Since $\h(\cdot,\xi)$ is $Y$-periodic for all $\xi\in\MM^{m\times d}$, $\h_n(\cdot,\xi)$ is $Y$-periodic for all $n\in\NN^*$ and all $\xi\in\MM^{m\times d}$.

Let $n\in\NN^*$, we will show that $\h_n\le \h_{n+1}$. By $\lsp$, we have $t_n\overline{\C}\subset t_{n+1}\overline{\C}\subset\inte\C$. Let $\xi\in\MM^{m\times d}$ and $x\in\RR^d$,
\begin{enumerate}[label=$\diamond$, ref=$\diamond$]
\item if $\xi\in t_n\overline{\C}$ then $\h_n(x,\xi)=\h_{n+1}(x,\xi)=\h(x,\xi)$;
\item if $\xi\in t_{n+1}\overline{\C}\setminus t_n\overline{\C}$ then, by Lemma \ref{decomposition} \ref{decomposition5}, we obtain \[
\h_n(x,\xi)=n(1+\dist(\xi,\overline{\C}))=n\le \h(x,\xi)=\h_{n+1}(x,\xi);
\]
\item if $\xi\notin  t_{n+1}\overline{\C}$ then 
\[
\h_n(x,\xi)=n(1+\dist(\xi,\overline{\C}))\le (n+1)(1+\dist(\xi,\overline{\C}))=\h_{n+1}(\xi)\le \h(x,\xi).
\]
\end{enumerate}
Hence $\h_n\le \h_{n+1}\le \h$ and then $\H\h_n\le\H\h_{n+1}\le \H\h$. 

Thus $\{\H\h_n\}_{n\in\NN^*}$ is a nondecreasing sequence satisfying
\begin{equation}\label{eq1 sup-quasiconvex 1}
\H\h_n(\xi)\le {\alpha}_n(1+\vert\xi\vert)\;\hbox{ and }\;\H\h_n(\xi)\le\H\h(\xi)
\end{equation}
for all $\xi\in\MM^{m\times d}$ and $n\in\NN^*$. We set $[\H\h]_\infty=\sup_{n\in\NN^*}\H\h_n$ which is lower semicontinuous since each $\H\h_n$ is continuous ($\H\h_n$ is finite and quasiconvex by Theorem \ref{BraidesTheorem}, then by Lemma \ref{fonseca} together with Lemma \ref{quasifinite} it follows that $\H \h_n$ is continuous).

\subsubsection{Approximation of $\H\h$}\label{approxfin}
Set $[\H\h]_\infty=\sup_{n\in\NN^*}\H\h_n$, then, by \eqref{eq1 sup-quasiconvex 1}, it holds that $[\H\h]_\infty\le \H{\h}$ and since $[\H\h]_\infty$ is quasiconvex
\begin{equation}\label{ineq1}
[\H\h]_\infty\le \Q\H\h\le \overline{\H\h}.
\end{equation}
Now, we will show that $[\H\h]_\infty\ge \overline{\H{\h}}$. Let $\xi\in\MM^{m\times d}$ and $k\in\NN^*$. Without loss of generality, we may assume that $[\H\h]_\infty(\xi)<+\infty$. Note that by a change of variable for every $n\in\NN^*$
\begin{equation}\label{changevariable}
\H\h_n(\xi)=\inf_{s\ge 1}\inf\left\{\int_Y \h_n(sx,\xi+\nabla\phi(x))dx:\phi\in W^{1,\infty}_0(Y;\RR^m)\right\}.
\end{equation}
By \eqref{changevariable}, it holds that for every $n\in\NN^*$ there exist $\phi_n^k\in W^{1,\infty}_0(Y;\RR^m)$ and $s_n^k\in\NN^*$ such that
\begin{align*}
{1\over 2k}+[\H\h]_\infty(\xi)\ge \int_{A_n^k} \h(s_n^kx,\xi+\nabla \phi_n^k(x))dx&+n\int_{Y\setminus A_n^k}\dist(\xi+\nabla \phi_n^k(x),\overline{\C})dx\\&+n\vert Y\setminus A_n^k\vert
\end{align*}
with $A_n^k=\{x\in Y:\xi+\nabla \phi_n^k(x)\in t_n\overline{\C}\}$. We deduce that
\begin{align}
&\lim_{n\to +\infty}\int_Y \dist(\xi+\nabla \phi_n^k(x),\overline{\C})dx=0,\label{convergence dist in L^1 1}\\
&\lim_{n\to +\infty}\vert Y\setminus A_n^k\vert=0.\label{convergence complem. A_n 1}
\end{align}
Since for any $\zeta\in\MM^{m\times d}$ it holds that $\vert\zeta\vert\le {\rm diam}(\overline{\C})+\dist(\zeta,\overline{\C})$, (\ref{convergence dist in L^1 1}) implies that $\{\xi+\nabla \phi_n^k\}_{n\in\NN^*}$ is bounded in $L^1(Y;\MM^{m\times d})$. Using Poincar\'e inequality and compact imbedding of $W^{1,1}(Y;\RR^m)$ in $L^1(Y;\RR^m)$ we deduce that there exists a subsequence (not relabeled) $\{\phi_n^k\}_{n\in\NN^*}$ converging in $L^1$. Applying Lemma \ref{Muller}, we can find a sequence $\{{\psi}_n^k\}_{n\in\NN^*}\subset W^{1,\infty}_0(Y;\RR^m)$ such that
\begin{align}
&\lim_{n\to +\infty}\left\vert\left\{ x\in Y: \nabla \phi_n^k(x)\ne\nabla {\psi}_n^k(x)\right\} \right\vert=0,\hbox{ and }\label{eq2.1 sup-quasiconvex1}\\
&\lim_{n\to +\infty}\Vert \dist(\xi+\nabla {\psi}_n^k,\overline{\C})\Vert_{\infty,Y}=0.\label{eq2.2 sup-quasiconvex1}
\end{align}
Let $\rho>0$ be such that $\rho\overline{B}\subset\inte\C$, where $\overline{B}=\{\xi\in\MM^{m\times d}:\vert\xi\vert\le 1\}$.
By \eqref{eq2.2 sup-quasiconvex1}, there exists $\sigma(k)\in\NN^*$ such that for every $n\ge \sigma(k)$
\[
\Vert \dist(\xi+\nabla\psi_{n}^k,\overline{\C})\Vert_{\infty,Y}\le {\rho\over 2k}.
\]
We construct $\sigma:\NN^*\to\NN^*$ in order to obtain an increasing function of $k$. By Lemma \ref{voisinage r}, we deduce $\xi+\nabla\psi_{n}^k(x)\in (1+{1\over k})\inte\C$ a.e. in $Y$ for all $n\ge \sigma(k)$. 

For each $l\in\NN^*$ we denote by $M_l(\cdot)=\sup\{\h(\cdot,\zeta):\zeta\in t_l\overline{\C}\}\in L^\infty_{\rm loc}(\RR^d)$ which is $Y$-periodic. Since Lemma \ref{decomposition} \ref{decomposition3}, we set 
\[
\overline{M}_l=\int_Y M_l(x)dx\le\Vert M_l\Vert_{\infty, Y}<+\infty.
\]
By \eqref{convergence complem. A_n 1} and \eqref{eq2.1 sup-quasiconvex1}, there exists $\delta(k)\ge 1$ such that for every $n\ge \delta(k)$
\[
\max\left\{\left\vert\left\{ x\in Y: \nabla \phi_{n}^k(x)\ne\nabla \psi_{n}^k(x)\right\} \right\vert,\vert Y\setminus A_{n}^k\vert\right\}\le {1\over 4 k \overline{M}_{\sigma(k)}}.
\]
Now, we take $n\ge \max\{\delta(k),\sigma(k)\}$ then
\[
\max\left\{\vert B_n^k\vert,\vert Y\setminus A_{n}^k\vert\right\}\le {1\over 4 k \overline{M}_{\sigma(k)}}\;\;\;\hbox{ and }\;\;\;\tau_k(\xi+\nabla\psi_{n}^k(x))\in\inte\C\hbox{ a.e. in } Y,
\]
where $B_n^k=\left\vert\left\{ x\in Y: \nabla \phi_{n}^k(x)\ne\nabla \psi_{n}^k(x)\right\} \right\vert$ and $\tau_k=\frac{k}{1+k}$. Set $G_n^k=B_n^k\cup Y\setminus A_n^k$, we have that $\vert G_n^k\vert\le (2k\overline{M}_{\sigma(k)})^{-1}$. Then it holds
\begin{align*}
&\int_Y \h(s_n^kx,t_{\sigma(k)}\tau_k(\xi+\nabla\psi_n^k(x)))dx\\
&= \int_{G_n^k} \h(s_n^kx,t_{\sigma(k)}\tau_k(\xi+\nabla\psi_n^k(x)))dx+\int_{(Y\setminus B_n^k)\cap A_n^k} \h(s_n^kx,t_{\sigma(k)}\tau_k(\xi+\nabla\psi_n^k(x)))dx\\
&\le\vert G_n^k\vert\int_Y\sup_{\zeta\in t_{\sigma(k)}\overline{\C}}\h(s_n^kx,\zeta)dx+\int_{A_n^k}\h(s_n^kx,t_{\sigma(k)}\tau_k(\xi+\nabla\phi_n^k(x)))dx\\
&\le\vert G_n^k\vert \overline{M}_{\sigma(k)}+\int_{A_n^k}\h(s_n^kx,t_{\sigma(k)}\tau_k(\xi+\nabla\phi_n^k(x)))dx\\
&\le {1\over 2k}+\int_{A_n^k}\h(s_n^kx,t_{\sigma(k)}\tau_k(\xi+\nabla\phi_n^k(x)))dx.
\end{align*}
By convexity of the distance function, we deduce from
(\ref{convergence dist in L^1 1}) that $\xi\in\overline{\C}$. The
$\lsp$ implies that $t_{\sigma(k)}\tau_k\xi\in\inte\C$. Using \eqref{changevariable}, we deduce that for every $k\in\NN^*$ and every $n\ge \max\{\delta(k),\sigma(k)\}$
\begin{equation}\label{eq fin}
\H\h(t_{\sigma(k)}\tau_k\xi)\le \frac{1}{2k}+\int_{A_{n}^k}\h(s_n^kx,t_{\sigma(k)}\tau_k(\xi+\nabla\phi_{n}^k(x)))dx.
\end{equation}
with $A_{n}^k=\{x\in Y:\xi+\nabla\phi_{n}^k(x)\in t_{n}\overline{\C}\}$.

Let $s\in\NN^*$. By (${\rm H}_1$), there exists $\eta_s>0$ such that for every $t\in [0,1[$, every $x\in\RR^d$ and every $\zeta\in\inte\C$ if $1-t\le \eta_s$ then $\h(x,t\zeta)\le \h(x,\zeta)+{1\over s}$. There also exists an integer $k_s\ge 1$ such that $1-t_{\sigma(k)}\tau_{k}\le\eta_s$ for all $k\ge k_s$ since $\sigma$ is increasing. Thus, if we take $k\ge k_s$ then for every $n\ge \max\{\sigma(k),\delta(k)\}$
\begin{align*}
\int_{A_{n}^k} \h(s_n^kx,t_{\sigma(k)}\tau_k(\xi+\nabla{\phi}_{n}^k(x)))dx&\le {1\over s}+\int_{A_{n}^k} \h(s_n^kx,\xi+\nabla{\phi}_{n}^k(x))dx\\
&\le {1\over s}+{1\over 2k}+[\H\h]_\infty(\xi).
\end{align*}
Hence, by \eqref{eq fin}
\begin{equation}\label{delta eps}
\H\h(t_{\sigma(k)}\tau_k\xi)\le{1\over s}+{1\over k}+[\H\h]_\infty(\xi).
\end{equation}
On one hand, by \eqref{ineq1}, letting $k\to +\infty$ and then $s\to +\infty$ 
\begin{equation*}
\Q\H\h(\xi)\le\overline{\H\h}(\xi)\le [\H\h]_\infty(\xi)\le \Q\H\h(\xi),
\end{equation*}
which implies 
\begin{equation}\label{equa1}
\overline{\H\h}=\Q\H\h=[\H\h]_\infty.
\end{equation}
On the other hand, by Corollary \ref{exist lim} and Remarks \ref{remarkwide} \ref{radialcontinuity}
\begin{align*}
\Q\H\h(\xi)\le\liminf_{k\to +\infty}\Q\H\h(t_{\sigma(k)}\tau_k\xi)&\le\widehat{\Z\H\h}(\xi)
=\liminf_{k\to +\infty}\widehat{\Z\H\h}(t_{\sigma(k)}\tau_k\xi)\\
&\le\liminf_{k\to +\infty}{\H\h}(t_{\sigma(k)}\tau_k\xi)\\
&\le {1\over s}+[\H\h]_\infty(\xi),
\end{align*}
which implies when $s\to +\infty$
\begin{equation}\label{equa2}
\Q\H\h\le \widehat{\Z\H\h}\le[\H\h]_\infty.
\end{equation}
Combining \eqref{equa1} and \eqref{equa2}
\begin{equation*}
\overline{\H\h}=\Q\H\h=\widehat{\Z\H\h}=[\H\h]_\infty.
\end{equation*}
The representation formula for $\widehat{\Z\H\h}$ is given by Corollary \ref{exist lim}. The proof is complete.\hfill$\square$
\bibliographystyle{alpha}

\end{document}